\documentclass[12pt]{article}
\usepackage[T1]{fontenc}
\usepackage{graphicx,amsmath,amsfonts,amssymb,amsthm,psfrag, xcolor,framed,thmtools,appendix,soul}
\usepackage{comment}
\usepackage{calc}
\usepackage{wrapfig,float}
\usepackage{fullpage}
\usepackage{setspace}

\newcommand{\dint}{\mathrm{d}}

\newcommand{\td}{\mathbb{T}}

\declaretheoremstyle[
spaceabove=6pt, spacebelow=6pt,
headfont=\normalfont\bfseries,
notefont=\mdseries, notebraces={(}{)},
headpunct=.\,— ,
bodyfont=\normalfont,
numbered=no
]{solu}
\declaretheorem[name=Commentaire, style=solu, preheadhook = \color{blue}]{Com}

\usepackage{comment}
\excludecomment{Com}

\excludecomment{comment} 

\newtheorem{thm}{Theorem}[section]
\newtheorem{corollary}[thm]{Corollary}
\newtheorem{lemma}[thm]{Lemma}
\newtheorem{proposition}[thm]{Proposition}
\newtheorem{assumption}{Assumption}[section]

\newtheorem{example}{Example}[section]
\newtheorem{rem}[thm]{Remark}
\numberwithin{equation}{section}
\newtheorem{assu}[thm]{Assumption}

\title{Exact simulation of the first passage time through a given level of jump diffusions.\protect\thanks{This study has been supported by the project PERISTOCH ANR-19-CE40-0023, 2020--2024 of the French National Research Agency (ANR). The IMB receives support from the EIPHI Graduate School (contract ANR-17-EURE-0002).}}

\begin{document}

\author{S. Herrmann and N. Massin\\[5pt]
\small{Institut de Math{\'e}matiques de Bourgogne (IMB) - UMR 5584, CNRS,}\\
\small{Universit{\'e} de Bourgogne Franche-Comt\'e, F-21000 Dijon, France} \\
\small{Samuel.Herrmann@u-bourgogne.fr}\\
\small{Nicolas.Massin@u-bourgogne.fr}
}
\maketitle

\begin{abstract}
Continuous-time stochastic processes play an important role in the description of random phenomena. It is therefore important to study particular stopping times dependent on the trajectories of these processes. Two approaches are possible: introducing an explicit expression of their probability distribution, and 
evaluating values generated by numerical models.
Choosing the second alternative, we propose an algorithm to generate the first passage time through a given level. The stochastic process under consideration is a one-dimensional jump diffusion that satisfies a stochastic differential equation driven by a Brownian motion. It is subject to random shocks that are characterised by an independent Poisson process. The proposed algorithm belongs to the family of rejection sampling procedures: the outcome of the algorithm and the stopping time under consideration are identically distributed. This algorithm is based on both the exact simulation of the diffusion value at a fixed time and on the exact simulation of the first passage time for continuous diffusion processes (see Herrmann and Zucca \cite{herrmann-zucca-exact}).
At a fixed point in time, the main challenge is to generate the position of a continuous diffusion that is conditioned not to reach a given level before that time. We present the construction of the algorithm and numerical examples. We also discuss specific conditions leading to the recurrence of a jump diffusion process.
\end{abstract}
\textbf{Key words and phrases:} first passage time, jump diffusion, Girsanov formula, rejection sampling, Bessel process, stochastic algorithm.\\
\textbf{2010 AMS subject classifications:} primary: 65C05; secondary: 60J60, 60J76, 60G40.
\section*{Introduction}
Precisely describing the time required by a diffusion process to the first overcome of a given threshold is a challenge in several fields, including  economics \cite{hu2012hitting}, finance \cite{janssen2013applied,Linetsky}, queueing, reliability theory \cite{pieper1997level},and neuroscience \cite{burkitt2006review, sacerdote2013stochastic}. The first passage time through a given level can be used, for example, to evaluate the risk of a  default in mathematical finance. Because this stopping time permits important decisions in practice, it is critical either to obtain an explicit expression of the corresponding probability distribution or to propose efficient algorithms for its random generation. An explicit expression of the density, often based on series expansions, is a source of valuable information. Unfortunately, this expression is only available for particular stochastic models and cannot be used for wide classes of applications. We thus propose a numerical approach, where the challenge is to generate the first passage times using efficient algorithms. 

In many applications, the behaviour of an experimental random value as time elapses can be modelled by a suitable one-dimensional diffusion process $(X_t,\,t\ge 0)$. The corresponding first passage time through a given threshold $L$ is defined by:
\begin{equation}\label{def:tau}
\tau_L:=\inf\{t\ge 0:\ X_t\ge L\}.
\end{equation}
For simplicity, we assume that the starting value $X_0=y_0<L$ is deterministic. In the continuous paths framework, diffusion is represented by the solution of a stochastic differential equation driven by Brownian motion. Different approaches have been proposed to generate $\tau_L$. One method is to use numerical approximations of the paths and to deduce the approximation of the passage time as a byproduct. Most existing studies have been based on improvements of the classical Euler scheme \cite{broadie1997continuity}, \cite{gobet2000weak}, \cite{gobet2010stopped}, \cite{bouchard-al}. In this framework of the time-stepping procedure, we also refer to \cite{kloedenplaten1999Euler}.

Another option is to use the exact simulation techniques introduced by Beskos, Papaspiliopoulos and Roberts and based on the Girsanov transformation. This approach consists of sophisticated rejection sampling. Let us recall the general rejection sampling method: independent and identically distributed (i.i.d) objects are sequentially proposed, and each object is accepted or rejected according to a probability weight. This weight only depends on the proposed object. The procedure stops as soon as an acceptance occurs, and the corresponding object is then chosen. This simple algorithm plays a critical role in the random generation framework. We let $\mathbb{P}$ be a probability measure that is absolutely continuous with respect to another probability measure $\mathbb{Q}$. We then assume that the Radon-Nikodym derivative $\frac{\dint \mathbb{P}}{\dint \mathbb{Q}}$ has an upper bound $c$. To generate a $\mathbb{P}$-distributed variable, independent $\mathbb{Q}$-distributed random variables $(R_n)_{n\ge 1}$ can be generated sequentially. Each variable is then accepted or rejected according to its value; the probability of acceptance is given by $\frac{1}{c}\frac{\dint \mathbb{P}}{\dint \mathbb{Q}}(R_n)$. This algorithm permits the generation of $\mathbb{P}$-distributed variates in a simple way provided that the Radon-Nikodym derivative has an explicit expression. Otherwise alternative methods are required to accept the random variates according to the correct Radon-Nykodym weight. 

Beskos, Papaspiliopoulos and Roberts used the classical rejection 
methodology for the simulation of particular random vectors associated with diffusion trajectories. The distribution of finite dimensional marginals of the Brownian paths is well-known, 
as is the exact distribution of their passage time. Thus, variables depending on the Brownian paths shall represent the proposal random objects in the rejection procedure. The proposal is accepted regarding the Radon-Nykodym derivative,
which connects both the diffusion and the Brownian laws of probability.
Using the Girsanov transformation permits us to describe the acceptance probability and therefore to clarify the methodology,
one generates a Brownian path (more precisely finite dimensional marginals) and accepts or rejects it with a probability depending on the entire path. Beskos, Papaspiliopoulos and Roberts \cite{bpr, beskos2008factorisation, beskos2005exact} first proposed such an approach to simulate variates that depend on the diffusion trajectory
moving on a given time interval $[0,\mathbb{T}]$. The procedure requires particular conditions on the drift and diffusion coefficients, but several improvements permit ease.
An interesting localisation argument was developed by Chen and Huang \cite{chen-huang}. Splitting the state space $\mathbb{R}$ into small intervals allows us to observe the diffusion paths between two exit times. Because the trajectory stays in a specific interval during this random time slot, the conditions can be weakened. This process does not occur without compensation: instead of using Brownian paths as proposal objects in rejection sampling, Chen and Huang consider Brownian meanders. The basic elements used in the algorithm are therefore more sophisticated because they also require a rejection procedure; the algorithm is built as a double layer rejection sampling. Thus, the simulation of first exit times appears to be a byproduct of the so-called localisation argument. Casella and Roberts \cite{casella2008exact} also extended the initial work of Beskos, Papaspiliopoulos and Roberts, and obtained the conditional distribution of the diffusion value at some fixed time given that the process never exits from a specific interval before. Both approaches are different and concern continuous diffusion processes on a fixed time interval $[0,\mathbb{T}]$. For efficiency purposes, we propose an algorithm that avoids space splitting, which can be extended to the jump diffusion framework and is not limited to finite time intervals. No existing algorithm satisfies these three properties.

In view of the description of first crossing times, an alternative approach was proposed by Herrmann and Zucca \cite{herrmann-zucca-exact}, who adapted the 
method of Beskos, Papaspiliopoulos and Roberts to generate $\tau_L$ in an exact way. The outcome of the algorithm has the same distribution as $\tau_L$, and there is no approximation error. Their goal was not to weaken the conditions of the diffusion under consideration but rather to deal directly with an unbounded time interval. The proposal variable in the rejection procedure was the Brownian first passage time through the given threshold (i.e. an inverse Gaussian variate). The weight used for the acceptance of the proposal depends on a suitable Bessel process, a convenient process for simulation purposes. Herrmann and Zucca also proposed algorithms based on the Girsanov transformation for exit time generation \cite{herrmann2020exact1, herrmann2020exact}.

All applications cannot be concerned with one-dimensional continuous diffusion processes. A natural extension is to study jump diffusion. These processes are driven by both Brownian motion and a Poisson random measure, and they are special cases of one-dimensional L\'{e}vy processes. In finance, for example, the changes in a stock exchange can be represented by jump diffusion. In this case, the jumps represent possible events that occur and produce strong impacts on the asset price \cite{kou2002jump}.

For these stochastic processes, several studies concern the approximation of the diffusion trajectory on some finite time interval. The simulation of the first passage time is described by a byproduct in the same way as the procedure used in the continuous case. In particular, the basic Euler method can be adapted to diffusions that are driven by Wiener processes and Poisson random measures. Platen \cite{Platen1}, Maghsoodi \cite{Maghsoodi1, Maghsoodi2} and Gardon \cite{gardon} introduced explicit time-discretization schemes based on the It\^{o}-Taylor expansion and obtained interesting convergence in mean square results. Several studies allows us to understand how to adapt the time grid to the jump times and thereby reduce approximation error. Other types of stochastic convergence have been analysed for the numerical approximation. Other challenging research topics exist; introducing schemes based on semi-implicit Euler-Maruyama methods  \cite{higham1,higham2} or Runge-Kutta methods \cite{buckwar} is of particular interest. The purpose of this study is not to write an exhaustive overview of the literature on numerical analysis because the proposed approach for the generation of first passage times is not based at all on an approximation procedure.Therefore, we refer to the monograph \cite{platenbook} and the references therein for additional information on approximation methods.

 The goal of this paper is to focus on the so-called exact simulation method. It is important to generalise the classical simulation procedure introduced by Beskos, Papaspiliopoulos and Roberts. As already presented in the continuous case, two different approaches can be chosen: one based on the localisation argument and using Brownian meanders, and the other based on Bessel processes. The first method has already been adapted to jump diffusions (see \cite{casella-roberts}, \cite{goncalves-roberts} and \cite{pollock-johansen-roberts}) and thus permits the exact generation of the first exit time from a given interval \cite{giesecke-smelov} using this double layer rejection method. The generation of meanders is nevertheless technical and leads to rather complex algorithms. We propose in this study to adapt the second methodology based on Bessel processes to jump diffusions. We thereby avoid both space splitting and generation of meanders. The new algorithm is thus simpler and requires fewer lines of code, making it more efficient but only available for the first passage time generation. 

This study considers the exact simulation of the first passage time $\tau_L$ defined by \eqref{def:tau}, where $(X_t,\,t\ge 0)$ stands for a jump diffusion process. The remainder of this paper is organised as follows. First, we end the introduction by defining  jump diffusion as the unique solution of a stochastic differential equation. In this situation, only a finite number of jumps are observed on a fixed time interval. Between two consecutive jumps, the path of diffusion is continuous. Thus the first overrun
of a threshold corresponds either to a jump or to a continuous crossing. Both possibilities should be considered in the algorithm. 
We recall the exact simulation techniques introduced for continuous paths in Section \ref{sec:stopdiff} and propose new proofs. Once the continuous crossing time is known, we must compare it with the family of jump times. If a jump time occurs first, then it critical to know the precise value of the diffusion just before it. 
We must generate the value of a continuous diffusion conditioned not to have reached a given level before.
In Section \ref{sec:algostopdiff}, propose an efficient algorithm to generate $\tau_L\wedge \mathbb{T}$, where $\tau_L$ is the first passage time of a continuous diffusion process and $\mathbb{T}>0$ is any fixed time. If we investigate jump diffusion processes, $\mathbb{T}$ is a possible value of the first jump time. To observe what occurs after the first jump, the strong Markov property of the process is considered: the key tool presented in Section 2 can be used sequentially to generate the first passage time $\tau_L$ of a jump diffusion $(X_t,\, t\ge 0)$ (see Theorem \ref{thm:SJD}).  
In Section \ref{sec:infinite}, we focus on the particular situation when $\tau_L<\infty$. Such a context requires additional conditions on the process but also permits direct simulation of $\tau_L$ in an exact way. The last section describes numerical results.

\subsubsection*{Jump diffusion: definition and model reduction}
\begin{Com}
Je supprime ici les sous-parties puisque Lamperti devient une remarque
\end{Com}
Typically, and historically, jump diffusions are described in the following way. We consider a filtered   
probability space $(\Omega, \mathcal{F}, \underline{\mathcal{F}},\mathbb{P})$,  where $\underline{\mathcal{F}}=(\mathcal{F})_{t\ge 0}$ is a filtration, 
and a mark space $(\mathcal{E}, B(\mathcal{E}))$ where $\mathcal{E} \subset \mathbb{R}\setminus \{0\}$. This mark space can be interpreted as the space of jump amplitudes. We define an $\underline{\mathcal{F}}$-adapted Poisson measure 
on $\mathcal{E}\times [0,\infty)$ 
that is denoted by $p_{\phi}(dv \times dt)$, whose intensity measure is given by $\phi(dv)dt$, where $\phi$ is a non negative finite measure defined on $\mathcal{E}$ with $\lambda := \phi(\mathcal{E})$. 

A jump diffusion  $X$ with jump rate $j$ is defined as follows:
\begin{equation}
dX_t =  \mu(t,X_{t-}) \,dt + \sigma(t,X_{t-}) \,dB_t + \int_{\mathcal{E}}j(t,X_{t-},v)p_{\phi}(dv \times dt),\quad t\ge 0,
\label{jdeq}
\end{equation}
with the initial value $X_0=y_0$.
In this study, $(B_t,\, t\ge 0)$ stands for a one-dimensional Brownian motion. The Poisson measure is assumed to be homogeneous in this study; the function $\phi$ only depends on the space variable. The result can thus be extended to the time-dependent case using a thinning procedure (see \cite{casella2011exact} for a description of this general approach). 

Classic conditions on the coefficients ensure the existence of a unique strong solution of \eqref{jdeq}, whose trajectories are right-continuous with left-hand limits (see Theorem 9.1 in \cite{ikeda-watanabe} for homogeneous coefficients and Theorem 1.19 in \cite{oksendal2005applied} in a general non-homogeneous case). These conditions introduced in Assumption \ref{assu:Wee1} concern the regularity of the coefficients and a growth property.We thus assume  that $\sigma(t,x)>0$ for all $(t,x)\in\mathbb{R}_+\times\mathbb{R}$ to avoid any degeneracy. 

For numerical studies, the representation \eqref{jdeq} is not convenient; we prefer another representation. The Poisson measure $p_\phi$ permits to generate a sequence of $\mathcal{E}\times [0,\infty)$-valued random points $(\xi_i,T_i)_{i\ge 1}$ where $(T_i)_{i\ge 1}$ is increasing.
The sequence of couples represents each jump amplitude and the corresponding jump time. The random amplitudes $(\xi_i)_{i\ge 1}$ are independent with distribution function $\phi/\lambda$.
%
The time spent between two consecutive jumps is exponentially distributed; therefore, we introduce $T_i=\sum_{k=1}^i E_k$, 
where $(E_k)_{k\ge 1}$ stands for exponentially distributed random variables with an average of $1/\lambda$. The initial position of the diffusion is given by $Y_0=y_0$. Between two jumps, the stochastic process satisfies the stochastic differential equation:
%
%
\begin{equation}
dY_{t} = \mu(t,Y_{t})\,dt + \sigma(t,Y_{t})\,dB_t, \quad \mbox{for}\ T_i<t<T_{i+1},\quad i\in\mathbb{N},
\label{afterijump}
\end{equation} 
and the jumps modify the trajectories as follows:
\begin{equation}
Y_{T_i}=Y_{T_{i}-}+j(T_i,Y_{T_i-},\xi_i),\quad \forall i\in\mathbb{N},
\label{afterijump1}
\end{equation}
where the sequence $(\xi_i)_{i\ge 1}$ must be independent of the  Brownian motion $(B_t)_{t\ge 0}$. The solution $(X_t)_{t\ge 0}$ of \eqref{jdeq} has the same path distribution as $(Y_t)_{t\ge 0}$, which is defined by \eqref{afterijump}--\eqref{afterijump1}.

The goal of this study is to simulate the first overrun of a given level $L$ for this jump diffusion.  For the proposed discussion, we  assume that $y_0<L$; it is straightforward to deduce the general case. The second representation plays a critical role; thus, we  simulate the trajectory of an SDE solution between two successive jumps exactly. The simulation of the first overrun should be based on the following intuitive procedure: we first simulate the trajectory of the stochastic process that satisfies the SDE without jump exactly  and consider its first passage time through the level $L$ denoted by $\tau_L$. Independently, we simulate the first exponentially distributed jump time $T_1$. If $\tau_L\le T_1$, then $\tau_L$ corresponds to the first passage time of the jump diffusion. In the other case, we simulate the position of the diffusion after the first jump using
\[
Y_{T_1}=Y_{T_1-}+j(T_1,Y_{T_1-},\xi_1).
\] 
We distinguish two likely different cases:  if $Y_{T_1}\ge L$, then $\tau_L=T_1$; otherwise, we know that $\tau_L>T_1$. The strong Markov property of the jump diffusion permits the start of a new jump diffusion with the initial position $Y_{T_1}<L$ and initial time $T_1$. Therefore, we repeat the procedure just presented, between $T_1$ and $T_2$, etc. An important tool for the simulation is the exact generation of continuous diffusion paths. In \cite{herrmann-zucca-exact}, the authors propose an efficient method based on both rejection sampling and the Girsanov transformation.


\begin{rem}
We assume that $\sigma$ is a $\mathcal{C}^{0,1}(\mathbb{R}_+\times\mathbb{R},\mathbb{R})$-continuous function such that $\sigma(t,x)>0$ for any $(t,x) \in \mathbb{R}_+\times \mathbb{R}$. 
To simplify the study,  we consider Lamperti's approach, which transforms the SDE:  
\begin{equation}
dX_t = \mu(t,X_t)dt + \sigma(t,X_t)dB_t,\quad t\ge 0,
\end{equation}
with initial Condition $X_0=x_0$ into the following equation:
\begin{equation}
dZ_t = \alpha(t,Z_t)dt + dB_t,\quad t\ge 0
\end{equation}
This well-known transformation corresponds to:
\begin{equation*}
Z_t =\nu(t,X_t) = \int_{0}^{X_t} \frac{1}{\sigma(t,x)}\,dx
\label{lamperti}
\end{equation*}
and:
\begin{equation*}
\alpha(t,x) = \frac{\partial \nu}{\partial t}(t,\nu^{-1}(t,x)) + \frac{\mu(t,\nu^{-1}(t,x))}{\sigma(t,\nu^{-1}(t,
x))} - \frac{1}{2}\frac{\partial \sigma}{\partial x}(t,\nu^{-1}(t,x)),
\end{equation*}
where $\nu^{-1}: \mathbb{R}_+\times\mathbb{R}\to\mathbb{R}$ is the unique function such that $\nu^{-1}(t,\nu(t,x)) = x$ for any $(t,x) \in \mathbb{R}_+\times \mathbb{R}$.
The jump times remain unchanged by the function $\nu$; we still consider the sequence $(T_i)_{i\ge 1}$, while the jump amplitudes are modified. Thus, the transport property permits us to obtain the following identity:
\begin{equation}
Z_{T_i}=Z_{T_i-}+\hat{\jmath}(t,Z_{T_i-},\xi),\quad i\in\mathbb{N},
\label{lamperti2}
\end{equation}
with
\[
\hat{\jmath}(t,z,v):=\nu(t,\nu^{-1}(t,z)+j(t,\nu^{-1}(t,z),v))-\nu(t,\nu^{-1}(t,z)).
\]
\end{rem}

\label{reduction}
\begin{Com}
Our aim is to deal with the general framework pointed out in the previous section. However we shall emphasize an interesting technique which permits to transform the considered stochastic differential equation into a simplified equation. Indeed Lamperti's transformation permits to change the equation in such a way that the diffusion coefficient becomes constant. This method is commonly used for  classical continuous diffusions and we just recall the crucial idea before presenting the extension to jump diffusions (usually the method is presented in the time-homogeneous context, we choose to present here diffusions with time dependent coefficients).
We consider the following SDE:
\begin{equation}
dX_t = \mu(t,X_t)dt + \sigma(t,X_t)dB_t,\quad t\ge 0,
\end{equation}
with the initial condition $X_0=x_0$.
The aim of Lamperti's transformation is to find a particular diffusion process $(Z_t)_{t\ge 0}$ defined as a functional of both $t$ and $X_t$,
that is  $Z_t = \nu(t,X_t)$, such that $Z_t$ satisfies 
\begin{equation}
dZ_t = \alpha(t,Z_t)dt + dB_t,\quad t\ge 0.
\end{equation}
Of course the function $\alpha$ has to be defined using $\nu$.
Applying It\^o's lemma to the process $Z_t$ leads to
\begin{align*}
dZ_t &= \frac{\partial \nu}{\partial t}(t,X_t)\, dt + \frac{\partial \nu}{\partial x}(t,X_t) \Big(\mu(t,X_t) dt  + \sigma(t,X_t) dB_t\Big) +\frac{1}{2}\frac{\partial^2 \nu}{\partial x^2}(t,X_t)\,d\langle X_t, X_t\rangle\\
&=\left( \frac{\partial \nu}{\partial t}(t,X_t) + \frac{\partial \nu}{\partial x}(t,X_t)\mu(t,X_t) + \frac{1}{2}\frac{\partial^2 \nu}{\partial x^2}(t,X_t) \sigma(t,X_t)^2\right)dt + \frac{\partial \nu}{\partial x}(t,X_t) \sigma(t,X_t)\,dB_t.
\end{align*}
From this equality, we deduce the importance to find $\nu$ such that $\frac{\partial \nu}{\partial x} =  \frac{1}{\sigma}$. We obtain the following statement:

\begin{proposition}
Let us assume that $\sigma$ is a $\mathcal{C}^{0,1}(\mathbb{R}_+\times\mathbb{R},\mathbb{R})$-continuous function. Let $\xi\in\mathbb{R}$. If $\sigma(t,x)>0$ for any $(t,x) \in \mathbb{R}_+\times \mathbb{R}$, then the process defined by
\begin{equation*}
Z_t =\nu(t,X_t) = \int_{\xi}^{X_t} \frac{1}{\sigma(t,x)}\,dx
\label{lamperti}
\end{equation*}
satisfies the following SDE:
\begin{equation}
dZ_t = \left(\frac{\partial \nu}{\partial t}(t,\nu^{-1}(t,Z_t)) + \frac{\mu(t,\nu^{-1}(t,Z_t))}{\sigma(t,\nu^{-1}(t,
Z_t))} - \frac{1}{2}\frac{\partial \sigma}{\partial x}(t,\nu^{-1}(t,Z_t))\right) dt + dB_t,\quad t\ge 0,
\label{edslamperti}
\end{equation}
with initial value $Z_0= \nu(0,x_0)$. Here $\nu^{-1}: \mathbb{R}_+\times\mathbb{R}\to\mathbb{R}$ is the unique function verifying $\nu^{-1}(t,\nu(t,x)) = x$ for any $(t,x) \in \mathbb{R}_+\times \mathbb{R}$.
\end{proposition}

\begin{proof}
As previously stated, the particular choice of the function $\nu$ permits to observe that 
\begin{equation*}
\frac{\partial \nu}{\partial x}(t,X_t) =  \frac{1}{\sigma(t,X_t)},
\end{equation*}
leading to
\begin{equation*}
\frac{\partial^2 \nu}{\partial x^2}(t,X_t) =  -\frac{\frac{\partial \sigma}{\partial x}(t,X_t)}{\sigma(t,X_t)^2}\quad\mbox{and}\quad\frac{\partial \nu}{\partial t}(t,X_t)=-\int_\xi^{X_t}\frac{\frac{\partial \sigma}{\partial t}(t,x)}{\sigma(t,x)^2}\,dx.
\end{equation*}
The announced result \eqref{edslamperti} is therefore an immediate consequence of It\^o's lemma.
\end{proof}

Let us just note that the general statement can be simplified in the particular homogeneous case. If neither the diffusion coefficient nor the drift term depend directly on the time variable, that is $\sigma(t,x)\equiv \sigma(x)$ and $\mu(t,x)\equiv\mu(x)$, then the transformation becomes
\begin{equation}
Z_t = \nu(X_t) = \int_{\xi}^{X_t} \frac{1}{\sigma(u)}du,
\end{equation}
where $\xi$ is a fixed value belonging to the state space. It is often convenient to choose $\xi$ as the starting position of the diffusion, the new process $(Z_t)_{t\ge 0}$ starts then at the origin. An other interesting choice may be $\xi=L$, the fixed level under observation in the description of the first passage time. The FP problem then consists of considering the first time the new diffusion process $(Z_t)_{t\ge 0}$ goes through the level $0$. 
The transformed diffusion is solution of the SDE $dY_t = \alpha(X_t)dt + dB_t$, where 
\begin{equation}
\alpha(x) = \frac{\mu \circ \nu^{-1}(x)}{\sigma \circ \nu^{-1}(x)}- \frac{1}{2} \sigma ' \circ \nu^{-1}(x),\quad \forall x\in\mathbb{R}.
\label{alpha}
\end{equation}
Let us now consider jump diffusions. The same Lamperti transformation can be applied. Let us consider $Z_t:=\nu(t,X_t)$ defined in \eqref{lamperti}, we observe that $\nu$ transforms the diffusion equation between two successive jumps \eqref{afterijump} into 
\begin{equation}
dZ_t=\alpha(t,Z_t)\,dt+dB_t,\quad \mbox{for}\ T_i<t<T_{i+1},\quad i\in\mathbb{N}.
\label{lamperti1}
\end{equation}
Of course the jump times are not changed at all by the function $\nu$: we still work with the sequence $(T_i)_{i\ge 1}$ whereas the jump amplitudes undergo modifications. Hence the transport property permits to modify \eqref{afterijump1} into
\begin{equation}
Z_{T_i}=Z_{T_i-}+\hat{\jmath}(t,Z_{T_i-},\xi),\quad i\in\mathbb{N},
\label{lamperti2}
\end{equation}
where 
\[
\hat{\jmath}(t,z,v):=\nu(t,\nu^{-1}(t,z)+j(t,\nu^{-1}(t,z),v))-\nu(t,\nu^{-1}(t,z)).
\]
Indeed it suffices to note that
\[
Z_{T_i}-Z_{T_i-}=\nu(T_i,X_{T_i-}+j(T_i,X_{T_i-},\xi_i))-\nu(T_i,X_{T_i-})\quad\mbox{and}\quad X_t=\nu^{-1}(t,Z_t).
\]
To conclude, the Lamperti transformation permits to change the jump diffusion \eqref{afterijump}--\eqref{afterijump1} into the jump diffusion \eqref{lamperti1}--\eqref{lamperti2}. In the following we shall assume that the diffusion coefficient is constant: $\sigma\equiv 1$, that corresponds to a model reduction procedure.
\end{Com}

\section{Simulation of the first passage time for a stopped continuous diffusion}
\label{sec:stopdiff}
We now fix a time parameter $\td>0$. In this section, we focus on continuous diffusion paths moving on the time interval $[0,\td]$. Generating a random object exactly means that we must generate an object using a stochastic algorithm such that both objects have the same distribution. Beskos, Papaspiliopoulos and Roberts \cite{bpr} proposed an efficient algorithm to simulate a continuous diffusion path exactly on the interval under consideration: $[0,\td]$. However, we cannot generate the entire paths numerically thus, the exact simulation effectively simulate a sequence of random points belonging to the trajectory in that case.

Herrmann and Zucca \cite{herrmann-zucca-exact} modified the algorithm introduced by Beskos, Papaspiliopoulos and Roberts to  generate $\tau_L$, the first passage time through level $L$ for continuous diffusion. Their approach presented in the next paragraph primarily differs from the localisation procedure introduced in \cite{chen-huang}, which is based on the generation of Brownian meanders. To address jump diffusion (next section), we also simulate the couple $(\tau_L\wedge \td, Y_{\tau_L\wedge \td})$, where $(Y_t)_{t\ge 0}$ represents continuous diffusion. The proposed technique also differs from the method proposed by Casella and Roberts \cite{casella2008exact}, who studied the exact simulation of killed diffusions introducing constrained Brownian bridges.

\subsection{Diffusion without any jump: the exact simulation method}
We first consider continuous diffusion processes and propose a numerical approach for the generation of their paths. We recall in this section the famous approach introduced by Beskos, Papaspiliopoulos and Roberts to set out the notations and we propose simplified proofs.
As already explained, we only investigate the reduced model in this study, and the generalisation was obtained by the Lamperti transformation. Thus, we consider, on a given probability space $(\Omega,\mathcal{F},\mathbb{P})$, the following SDE on the fixed time interval $[0,\td]$:
\begin{equation}
dY_t=\alpha(t,Y_t)\,dt+dB_t,\quad Y_0=y_0,
\label{numero}
\end{equation}
where $(B_t)_{t\ge 0}$ is a standard one-dimensional Brownian motion.
The direct generation of a diffusion path is a difficult task, which is why we introduce
the link between diffusion and standard Brownian motion using the famous Girsanov formula. The goal of this formula is to find a suitable probability space in which  the considered diffusion is Brownian motion. Then we generate a Brownian motion path and perform an acceptance/rejection methodology using the probability weight appearing in the Girsanov formula.
\subsubsection*{Girsanov transformation: consequences for simulation purposes}
First, we recall the statement of the Girsanov transformation and the associated Novikov condition (for a reference about Girsanov's transformation, see, for instance, \cite{Oksendal}).
\begin{assumption}[Novikov's condition]
The drift term $\alpha$ satisfies Novikov's condition if: 
\begin{equation}
\mathbb{E}_{\mathbb{P}}\left[\exp\left(\frac{1}{2}\int^\td_0 \alpha^2(s,y_0+B_s)\, ds\right)\right] < \infty.
\label{Novikov}
\end{equation}
\end{assumption}
This particular condition is satisfied as soon as the growth of the drift term $\alpha$ is at most linear (see Corollary 5.16 p.200 in \cite{K-S}). There exists a constant $C_\td>0$ such that:
\[
|\alpha(t,x)|\le C_\td(1+|x|),\quad \forall (t,x)\in[0,\td]\times \mathbb{R}.
\]
Then, the following transformation holds.
\begin{thm}
Assuming that $\alpha$ satisfies Novikov's condition, we define the martingale $(M_t)_{t\ge 0}$ by:
\begin{equation}
M_t = \exp\left(\int^t_0 \alpha(s,y_0+B_s)\,dB_s - \frac{1}{2}\int_0^t \alpha^2(s,y_0+B_s)ds\right), \quad t \leq \td,
\label{weight}
\end{equation}
and the measure $\mathbb{Q}$ on $(\Omega, \mathcal{F}_\td)$:
\begin{equation}
d\mathbb{Q}=M_\td\,d\mathbb{P}.
\end{equation}
Then, under $\mathbb{Q}$, the stochastic process  $\Big(B_t-\int_0^t\alpha(s,B_s)\,ds\Big)_{0\le t\le \td}$ is a one-dimensional standard Brownian motion.  $(y_0+B_t)_{0\le t\le \td}$ under $\mathbb{Q}$ has the same distribution as  $(Y_t)_{0\le t \le \td}$ under $\mathbb{P}$. \label{thm-RN}
\end{thm}
The Radon-Nikodym derivative introduced in the previous statement becomes the weight necessary for the use of an acceptance/rejection sampling.
We use the Girsanov formula as follows. We consider $(Y_t)_{t\ge 0}$, the solution of the SDE \eqref{numero}, and $f$ to be any measurable functional depending on the paths of $Y$ that are observed on the time interval $[0,\td]$. Then:
\begin{align*}
\mathbb{E}_{\mathbb{P}}[f(Y_{\cdot})] &=\mathbb{E}_{\mathbb{Q}}[f(y_0+B_{\cdot})] = \mathbb{E}_{\mathbb{P}}[f(y_0+B_{\cdot})\cdot M_\td]  \\
&= \mathbb{E}_{\mathbb{P}}\left[f(y_0+B_{\cdot})\exp\left(\int^\td_0 \alpha(s,y_0+B_s) dB_s - \frac{1}{2}\int_0^\td \alpha^2(s, y_0+B_s)ds\right)\right],
\end{align*}
where $(B_t)_{t\ge 0}$ is a standard Brownian motion under the probability measure $\mathbb{P}$. 
We introduce two different functions that play a critical role in the numerical algorithm. We define:
\begin{equation}
\beta(t,x):=\int_{y_0}^x \alpha(t,y)\,dy\quad \mbox{and}\quad \gamma(t,x)=\frac{\partial \beta}{\partial t}(t,x)+\frac{1}{2}\,\Big( \frac{\partial \alpha}{\partial x}(t,x)+\alpha^2(t,x) \Big).
\label{eq:def:beta-gamma}
\end{equation}
Using It\^o's formula applied to the stochastic process $(\beta(t,y_0+B_t))_{t \geq 0}$, we obtain: 
\begin{equation}
\mathbb{E}_{\mathbb{P}}[f(Y_{\cdot})] =\mathbb{E}_{\mathbb{P}}[f(y_0+B_{\cdot})\cdot \hat{M}_\td]\quad \mbox{with}\quad\hat{M}_\td:=e^{\beta(\td,y_0+B_\td) - \int^\td_0  \gamma(t,y_0+B_s)\,ds}.
\label{eq:Girsanov}
\end{equation}
\subsubsection*{General hypotheses for continuous diffusion processes}
Throughout this study, diffusion \eqref{numero} begins in $y_0$ with $y_0<L$. We now present different hypotheses concerning the drift coefficient $\alpha$ in \eqref{numero} that permit us to describe a typical framework for the development of efficient algorithms. The goal of this process is not, at this stage, to precisely emphasise the most general situation that permits the use of the exact simulation technique; several studies have already  successively weaken these conditions (see Beskos, Papaspiliopoulos and Roberts \cite{bpr}). The goal of this process is rather to consider a convenient context where the critical arguments used in the study of the continuous diffusions can easily
be adapted to jump diffusions. First, we require a classical regularity property to apply It\^o's formula.
\begin{assumption}
The drift coefficient $\alpha$ is a $\mathcal{C}^{1,1}(\mathbb{R}_+\times\mathbb{R})$-continuous function.
\label{assum20}
\end{assumption}
\noindent The regularity property of $\alpha$ immediately implies that $\beta$ and $\gamma$ defined by \eqref{eq:def:beta-gamma} are well-defined and continuous functions. According to the simulation goals, we need additional conditions such as the boundedness of $\beta$ or of $\gamma$. If the goal is to generate $Y_\td$, we require the following assumption.
\begin{assumption}
The function $\gamma$ defined in \eqref{eq:def:beta-gamma} is nonnegative and satisfies the following; thus, there exists a constant $\kappa$ s.t.:
\begin{equation}
0 \leq \gamma(t,x) \leq \kappa, \quad \forall (t,x) \in[0,\td]\times\mathbb{R}.
\end{equation}
\label{assum21}
\end{assumption}
\begin{assumption}
The function $\beta$ defined in \eqref{eq:def:beta-gamma} has an upper bound at time $\td$; thus, there exists a constant $\beta_+>0$ s.t.
\begin{equation}
\beta(\td,x) \leq \beta_+, \quad \forall x \in\mathbb{R}.\label{assum22-1}
\end{equation}
\label{assum22}
\end{assumption}
\noindent This assumption can be weakened. Because condition \eqref{assum22-1} concerns the indefinite integral of the drift term $\alpha(\cdot,\cdot)$ with respect to the space variable, classical drift terms such as constant functions do not satisfy it. It is thus important to propose another type of condition: the existence of an upper bound is replaced by an integration property. Such a condition is in particular satisfied by constant functions $\alpha(\cdot,\cdot)$.
\begin{assumption}
The function $\Gamma_\td:\mathbb{R}\to\mathbb{R}$ defined by:
\[
\Gamma_\td(x):=
\exp\Big\{\beta(\td,y_0+x)-\frac{x^2}{2\td}\Big\},\quad x\in\mathbb{R},
\]
with $\beta$ introduced in \eqref{eq:def:beta-gamma}, is integrable: $\Gamma_\td\in L^1(\mathbb{R})$.
\label{assum23}
\end{assumption}
\noindent  If $\beta$ satisfies Assumption \ref{assum22}, then the corresponding function $\Gamma_\td$ is integrable with respect to the Gaussian measure in the sense of Assumption \ref{assum23}.   
Assumptions \ref{assum21}, \ref{assum22} and \ref{assum23} essentially concern the simulation of $Y_\td$, where $\td$ is a fixed time. If we are rather interested in the first passage time through level $L$, we shall focus the proposed attention on the space subset $]-\infty,L]$.
\begin{assumption}
The function $\gamma$ defined in \eqref{eq:def:beta-gamma} is nonnegative and satisfies the following: there exists a constant $\kappa$ s.t.:
\begin{equation}
0 \leq \gamma(t,x) \leq \kappa, \quad (t,x) \in \mathbb{R}_+\times ] -\infty , L].
\end{equation}
\label{assum24}
\end{assumption}
\begin{assumption}
The function $\beta$ defined in \eqref{eq:def:beta-gamma} has an upper bound; thus, there exists a constant $\beta_+>0$ s.t.
\begin{equation}
\beta(t,x) \leq \beta_+, \quad \forall (t,x) \in\mathbb{R}_+\times ]-\infty,L].
\end{equation}
\label{assum25}
\end{assumption}
\noindent All these assumptions change as soon as the drift coefficient $\alpha$ in equation \eqref{numero} is homogeneous in time. The results presented in most previous studies (\cite{bpr}, \cite{herrmann-zucca-exact}) concern this restrictive context, but generalisation is a quite simple task.
%
%
\subsubsection*{Approach developed by Beskos, Papaspiliopoulos and Roberts}
Beskos, Papaspiliopoulos and Roberts proposed in \cite{bpr} a simulation methodology for the exact generation of diffusion paths on some given interval $[0,\mathbb{T}]$. Their study is based on the Girsanov transformation and on acceptance/rejection sampling. As already explained, to generate a random object with probability distribution $\mathbb{P}$, an iterative procedure is introduced based on the generation of independent $\mathbb{Q}$-distributed random objects. Then, their acceptance or rejection associated with a weight proportional to $\frac{\dint \mathbb{P}}{\dint\mathbb{Q}}$ achieves the intended goal. This Radon-Nikodym derivative in the diffusion case corresponds to $M_\mathbb{T}$, as shown in Theorem \ref{thm-RN}. This weight depends on the entire Brownian trajectory; rejection sampling therefore requires a sophisticated approach to accept or reject in accordance
with the correct probability. 

This procedure requires a Poisson process that is independent of diffusion \eqref{numero} that allows us  to obtain the required weight in the rejection method \eqref{eq:Girsanov}. Their method is not easy to adapt to jump diffusions because they do not use the Markov property of the diffusion process. Thus, we propose an alternative presentation of their result (and the corresponding proof) and thus start generalising the jump diffusion context.
For a clear and succinct presentation of this issue, we prefer to only introduce the exact simulation of $Y_\mathbb{T}$, where $(Y_t)_{t\ge 0}$ corresponds to the solution of \eqref{numero}.
\begin{framed}
\centerline{\sc Exact simulation of $Y_\td$ -- Algorithm $(B\!R)_\mathbb{T}^1$}
\emph{\begin{enumerate}
\item Let $(G_n)_{n\ge 1}$  be independent random variables with Gaussian distribution $\mathcal{N}(0,1)$
\item Let $(E_n)_{n\ge 1}$  be independent exponentially distributed r.v. with an average of $1/\kappa$. 
\item Let $(U_n)_{n\geq 1}$ be independent uniformly distributed random variables on $[0,1]$. 
\end{enumerate}
 The sequences  $(G_n)_{n\ge 1}$, $(E_n)_{n\ge 1}$ and $(U_n)_{n\geq 1}$ are assumed to be independent.}\\[5pt]
\noindent {\bf Initialisation:} $n=0$.\\[2pt]
{\bf Step 1.} Set $Z=y_0$, $\mathcal{T}=0$.\\[5pt]
{\bf Step 2.} While $\mathcal{T}<\mathbb{T}$ do:
\begin{itemize}
\item set  $n\leftarrow n+1$
\item $Z\leftarrow Z+\sqrt{\min(\mathcal{T}+E_n,\td)-\mathcal{T}}\,G_n$ and $\mathcal{T}\leftarrow \min(\mathcal{T}+E_n,\mathbb{T})$ 
\item If ($\mathcal{T}<\mathbb{T}$ and $\kappa\, U_n<\gamma(\mathcal{T}, Z)$) then go to Step 1.
\item If ($\mathcal{T}= \mathbb{T}$ and $U_ne^{\beta_+}>e^{\beta(\td,Z)}$) then go to Step 1.
\end{itemize}
{\bf Outcome:} the random variable $Z$. 
\end{framed}
\noindent Explanation of the procedure: Each iteration of \emph{Step 1} corresponds to the start of a new random experiment, which is independent of all previous experiments and consists of a generation of Brownian finite dimensional marginals represented by successive values of $Z$. The last value, which corresponds to the Brownian position at time $\mathbb{T}$, becomes the outcome of the algorithm, provided that all marginals satisfy specific constraints. The probability that all these conditions are simultaneously satisfied is exactly the weight proportional to $M_\mathbb{T}$ as expected. The objective is achieved: the value of the Brownian motion at time $\mathbb{T}$ is accepted with probability $M_\mathbb{T}$, and the outcome consequently has the same distribution as $Y_\mathbb{T}$.
\begin{proposition} Under Assumptions \ref{assum20}, \ref{assum21} and \ref{assum22}, both the outcome $Z$ of Algorithm $(B\!R)_\mathbb{T}^1$ and $Y_\mathbb{T}$, the time $\td$ of the diffusion process \eqref{numero}, have the same distribution.\label{prop:BR1}
\end{proposition}
\begin{rem} An adaptation of Algorithm $(B\!R)_{\mathbb{T}}^1$ should permit us to obtain more than just the random variable $Z$, which has the same distribution as  $Y_\mathbb{T}$. Denoting $n_1$ by $n$  throughout the last iteration of step number one, $n_2$ as the increment variable when the algorithm stops, and $Z_n$ (respectively $\mathcal{T}_n$) the successive values of $Z$ (resp. $\mathcal{T}$), we obtain that: 
\[
\Big\{(0,y_0),\, (\mathcal{T}_{n_1+1},Z_{n_1+1} ),\ldots,(\mathcal{T}_{n_2-1},Z_{n_2-1}),\,(\mathbb{T},Z_{n_2})\Big\}
\] 
is a set of points that has the same distribution as points belonging to the diffusion trajectory.
\label{rem:algo-BR}
\end{rem}
\begin{proof} We denote the number of necessary repetitions of \emph{Step 1} as $\mathcal{N}$ and let $\psi$ be a nonnegative measurable function. Because the algorithm is based on acceptance/rejection sampling (see Proposition \ref{prop:rejet}),
we obtain:
\begin{align*}
\mathbb{E}[\psi(Z)]&=\frac{\mathbb{E}[\psi(Z)1_{\{ \mathcal{N}=1 \}}]}{\mathbb{P}(\mathcal{N}=1)}=\frac{\nu(\psi)}{\nu(1)}\quad \mbox{where}\quad \nu(\psi):=\mathbb{E}[\psi(Z)1_{\{ \mathcal{N}=1 \}}].
\end{align*}
We now use the notations $(\mathcal{T}_n)_{n\ge 0}$ and $(Z_n)_{n\ge 0}$ introduced in Remark \ref{rem:algo-BR} and define the event $A_n:=\{\mathcal{T}_{n-1}<\mathbb{T}\le \mathcal{T}_{n-1}+E_{n}=\mathcal{T}_{n}\}$, for any $n\ge 1$, which describes the number of random intervals necessary to cover $[0,\mathbb{T}]$. We observe that:
\begin{align*}
\nu(\psi)&=\sum_{n\ge 1}\mathbb{E}\Big[\psi(Z_n)1_{\{\kappa\, U_1>\gamma(\mathcal{T}_1,Z_1),\ldots,\,  \kappa\, U_{n-1}>\gamma(\mathcal{T}_{n-1},Z_{n-1})\}}1_{\{ U_n\le e^{\beta(\mathbb{T},Z_n)-\beta_+} \}}1_{A_n}\Big].
\end{align*}
Taking the integral with respect to the independent uniform variates $(U_n)_{n\ge 1}$ leads to:
\begin{align*}
\nu(\psi)&=\sum_{n\ge 1}\kappa^{1-n}\mathbb{E}\Big[\psi(Z_n)(\kappa-\gamma(\mathcal{T}_1,Z_1))\,\ldots\,  (\kappa-\gamma(\mathcal{T}_{n-1},Z_{n-1})) e^{\beta(\mathbb{T},Z_n)-\beta_+}1_{A_n}\Big].
\end{align*}
We note that given $A_n$, $(\mathcal{T}_1,\ldots,\mathcal{T}_{n-1})$ has the same distribution as $(V^{(1)},\ldots, V^{(n-1)})$, an ordered $(n-1)$-tuple of uniform random variables $(V_1,\ldots, V_{n-1})$ on $[0,\mathbb{T}]$. The probability of event $A_n$ can be computed using a Poisson distribution of parameter $\kappa \mathbb{T}$. Finally, $(Z_1,\ldots Z_n)$ is a Gaussian vector and  has the same distribution as: \[(y_0+B_{V^{(1)}},\ldots, y_0+B_{V^{(n-1)}},y_0+B_\mathbb{T})\] where $(B_t)_{t\ge 0}$ is a standard Brownian motion independent of the $(n-1)$-tuple $(V_1,\ldots, V_{n-1})$ because the Brownian motion has Gaussian independent increments. Thus,
\begin{align*}
\nu(\psi)&=\sum_{n\ge 1}\mathbb{E}\Big[\psi(Z_n)(\kappa-\gamma(\mathcal{T}_1,Z_1))\,\ldots\,  (\kappa-\gamma(\mathcal{T}_{n-1},Z_{n-1})) e^{\beta(\mathbb{T},Z_n)-\beta_+}\Big\vert A_n\Big]\frac{\mathbb{T}^{n-1}}{(n-1)!}\,e^{-\kappa \mathbb{T}}\\
&=\sum_{n\ge 1}\mathbb{E}\Big[\psi(y_0+B_\mathbb{T})\prod_{j=1}^{n-1}(\kappa-\gamma(V^{(j)},y_0+B_{V^{(j)}}) e^{\beta(\mathbb{T},y_0+B_\mathbb{T})-\beta_+}\Big]\frac{\mathbb{T}^{n-1}}{(n-1)!}\,e^{-\kappa \mathbb{T}}\\
&=\sum_{n\ge 1}\mathbb{E}\Big[\psi(y_0+B_\mathbb{T})\prod_{j=1}^{n-1}(\kappa-\gamma(V_j,y_0+B_{V_j})) e^{\beta(\mathbb{T},y_0+B_\mathbb{T})-\beta_+}\Big]\frac{\mathbb{T}^{n-1}}{(n-1)!}\,e^{-\kappa \mathbb{T}}.
\end{align*}
Taking the expectation with respect to the uniformly distributed variates $V_j$, we have:
\begin{align*}
\nu(\psi)&=\sum_{n\ge 1}\mathbb{E}\Big[\psi(y_0+B_\mathbb{T})\Big(\kappa-\frac{1}{\mathbb{T}}\int_0^\mathbb{T}\gamma(s,y_0+B_{s})\,ds\Big)^{n-1} e^{\beta(\mathbb{T},y_0+B_\mathbb{T})-\beta_+}\Big]\frac{\mathbb{T}^{n-1}}{(n-1)!}\,e^{-\kappa \mathbb{T}}\\
&=\mathbb{E}\Big[\psi(y_0+B_\mathbb{T}) \exp\Big\{\beta(\mathbb{T},y_0+B_\mathbb{T})-\beta_+-\int_0^\mathbb{T}\gamma(s,y_0+B_{s})\,ds\Big\}\Big]\\
&=\mathbb{E}[\psi(y_0+B_\mathbb{T})\cdot\hat{M}_\mathbb{T}]\, e^{-\beta_+},
\end{align*}
where $(\hat{M}_t)_{t\ge 0}$ is the martingale defined in \eqref{eq:Girsanov}. The martingale property leads to $\nu(1)=\mathbb{E}[\hat{M}_0]e^{-\beta_+}=e^{-\beta_+}$. The Girsanov transformation permits us to conclude the following:
\[
\mathbb{E}[\psi(Z)]=\frac{\nu(\psi)}{\nu(1)}=\mathbb{E}[\psi(y_0+B_\mathbb{T})\cdot\hat{M}_\mathbb{T}]=\mathbb{E}[\psi(Y_\mathbb{T})].
\]
\end{proof}

\begin{rem} Algorithm $(B\!R)_\mathbb{T}^1$ permits the generation of the random variable $Y_\mathbb{T}$. Both parameters $\kappa$ and $\beta_+$ in the procedure do not affect the probability distribution of the outcome, but they strongly influence the algorithm efficiency. The previous proof indicates that the number of iterations $\mathcal{N}$ of \emph{Step 1} is geometrically distributed with average $\frac{1}{\nu(1)}=e^{\beta_+}$. Each random experiment corresponding to \emph{Step 2} is based on a random number of tests. This number, which roughly represents the computing cost of the experiment, is stochastically smaller than a Poisson random variable with parameter $\kappa \mathbb{T}$. Using the Wald identity (see Appendix) leads to the following upper bound of the overall cost of Algorithm $(B\!R)_\mathbb{T}^1$: $\kappa\mathbb{T}e^{\beta_+}$. It is therefore critical to choose the smallest values for both parameters $\kappa$ and $\beta_+$.\label{rem:eff}
\end{rem}

Under Assumptions \ref{assum20}, \ref{assum21} and \ref{assum22}, the algorithm considered in Proposition \ref{prop:BR1} has a convenient and intuitive expression. However, the boundedness of the function $\beta(\mathbb{T},\cdot)$ is restrictive, making it  important to propose an alternative approach. Thus, Beskos, Papaspiliopoulos and Roberts introduced an integrability condition for $\Gamma_\mathbb{T}$, written in Assumption \ref{assum23}. Because $\Gamma_\mathbb{T}$ is a nonnegative function, the integrability condition summarised in the identity $\Gamma_\mathbb{T}(\mathbb{R}):=\int_{\mathbb{R}}\Gamma_\mathbb{T}(x)\,dx<\infty$ ensures that $\Gamma_\mathbb{T}(\cdot)/\Gamma_\mathbb{T}(\mathbb{R})$ is a probability density function. This critical property allows us to present the following algorithm, where $x_+:=\min(x,0)$. 
\begin{framed}
\centerline{\sc Exact simulation of $Y_\mathbb{T}$ -- Algorithm $(B\!R)_\mathbb{T}^2$}
\emph{\begin{enumerate}
\item Let $(R_n)_{n\ge 1}$ be independent random variables with density $\Gamma_\mathbb{T}(\cdot)/\Gamma_\mathbb{T}(\mathbb{R})$.
\item Let $(G_n)_{n\ge 1}$  be independent random variables with Gaussian distribution $\mathcal{N}(0,1)$
\item Let $(E_n)_{n\ge 1}$  be independent exponentially distributed r.v. with average $1/\kappa$.
\item Let $(U_n)_{n\geq 1}$ be independent uniformly distributed  random variables on $[0,1]$. 
\end{enumerate}
 The sequences  $(R_n)_{n\ge 1}$, $(G_n)_{n\ge 1}$, $(E_n)_{n\ge 1}$ and $(U_n)_{n\geq 1}$ are assumed to be independent.}\\[5pt]
\noindent {\bf Initialisation:} $k=0$, $n=0$.\\[2pt]
{\bf Step 1.} Set $k\leftarrow k+1$ then $Z=y_0$, $W=y_0+R_k$ and $\mathcal{T}=0$.\\[5pt]
{\bf Step 2.} While $\mathcal{T}<\mathbb{T}$ do:
\begin{itemize}
\item set  $n\leftarrow n+1$
\item $\displaystyle Z\leftarrow Z+\frac{E_n}{\mathbb{T}-\mathcal{T}}\,W+\sqrt{\frac{E_n(\mathbb{T}-\mathcal{T}-E_n)_+}{\mathbb{T}-\mathcal{T}}}\ G_n$ and $\mathcal{T}\leftarrow \min(\mathcal{T}+E_n,\mathbb{T})$ 
\item If ($\mathcal{T}<\mathbb{T}$ and $\kappa\, U_n<\gamma(\mathcal{T}, Z)$) then go to Step 1.
\end{itemize}
{\bf Outcome:} the random variable $W$. 
\end{framed}
This procedure in this second algorithm is different from the first one. The critical idea of Algorithm $(B\!R)_\mathbb{T}^1$ is to simulate a Brownian motion on the interval $[0,\mathbb{T}]$ and to accept or reject the trajectory using the probability of selection issued from the Girsanov transformation. Acceptance strongly depends on the entire path of the process and leads to the outcome $y_0+B_\mathbb{T}$, the endpoint of the Brownian path. In Algorithm $(B\!R)_\mathbb{T}^2$, the approach is different: we consider a random variable $W$ with the proposal distribution $\Gamma_\mathbb{T}(\cdot)/\Gamma_\mathbb{T}(\mathbb{R})$ translated by $y_0$. This variate shall be accepted or rejected using a probability based on the entire path of a Brownian bridge starting in $y_0$ and ending at time $\mathbb{T}$ with $W$. The primary difference is therefore to replace the Brownian motion by the Brownian bridge. We thus obtain the following statement.
\begin{proposition} Under Assumptions \ref{assum20}, \ref{assum21} and \ref{assum23}, both the outcome $W$ of Algorithm $(B\!R)_\mathbb{T}^2$ and $Y_\mathbb{T}$, the value at time $\mathbb{T}$ of the diffusion process \eqref{numero}, have the same distribution.\label{prop:BR2}
\end{proposition}
The proof of Proposition \ref{prop:BR2} is similar to the proof of Proposition \ref{prop:BR1}. We refer to \cite{these-Nicolas} for the adapted proof. The conditions concerning the drift coefficient of the diffusion \eqref{numero} can be weakened using the so-called localisation procedure \cite{chen-huang}, but we do not intend to develop such an approach in this study as already mentioned.
\begin{Com}
\begin{proof}[Sketch of proof] The proof of Proposition \ref{prop:BR2} is quite similar to the proof of Proposition \ref{prop:BR1}.  We shall therefore not go into all details. Let $\psi$ a non-negative measurable function. Since the algorithm is based on a rejection sampling, we get as usual
\begin{align*}
\mathbb{E}[\psi(W)]&=\frac{\mathbb{E}[\psi(W)1_{\{ \mathcal{N}=1 \}}]}{\mathbb{P}(\mathcal{N}=1)}=\frac{\nu(\psi)}{\nu(1)}\quad \mbox{where}\quad \nu(\psi):=\mathbb{E}[\psi(W)1_{\{ \mathcal{N}=1 \}}],
\end{align*}
and $\mathcal{N}$ stands for the number of visits of the steps number one before the algorithm stops. Using similar notations and arguments as those developed in the previous proof, we observe that
\begin{align*}
\nu(\psi)&=\sum_{n\ge 1}\kappa^{1-n}\mathbb{E}\Big[\psi(W)(\kappa-\gamma(\mathcal{T}_1,Z_1))\,\ldots\,  (\kappa-\gamma(\mathcal{T}_{n-1},Z_{n-1}))1_{A_n}\Big].
\end{align*}
Given both $A_n$ and the value $W$, the $(n-1)$-tuple $(\mathcal{T}_1,\ldots,\mathcal{T}_{n-1})$ has the same distribution as $(V^{(1)},\ldots, V^{(n-1)})$ an ordered $(n-1)$-tuple of uniform random variables on $[0,\mathbb{T}]$. Moreover the probability of the event $A_n$ is linked to the Poisson distribution of parameter $\kappa \mathbb{T}$. Finally $(Z_1,\ldots Z_{n-1})$ is a Gaussian vector and has the same distribution as \[(b_{V^{(1)}},\ldots,b_{V^{(n-1)}})\] with $(b_t)_{0\le t\le \td}$ a Brownian bridge independent of the $(n-1)$-tuple $(V_1,\ldots, V_{n-1})$. The Brownian bridge starts for $t=0$ with the value $y_0$ and ends with the value $W$ at time $\mathbb{T}$. Hence
\begin{align*}
\nu(\psi)&=\sum_{n\ge 1}\mathbb{E}\Big[\psi(W)\prod_{j=1}^{n-1}(\kappa-\gamma(V_j,b_{V_j})) \Big]\frac{\mathbb{T}^{n-1}}{(n-1)!}\,e^{-\kappa \mathbb{T}}\\
&=\mathbb{E}\Big[\psi(W) \exp\Big\{-\int_0^\mathbb{T}\gamma(s,b_{s})\,ds\Big\}\Big]\\
&=\mathbb{E}\Big[\psi(W)\,\mathbb{E}\Big[ \exp\Big\{-\int_0^\mathbb{T}\gamma(s,y_0+B_{s})\,ds\Big\}\Big|y_0+B_\mathbb{T}=W\Big]\Big],
\end{align*}
where $(B_t)_{t\ge 0}$ is a standard Brownian motion. Using the explicit distribution of the variable $W$, we obtain
\begin{align*}
\nu(\psi)&=\frac{1}{\Gamma_\mathbb{T}(\mathbb{R})}\,\int_{\mathbb{R}}\psi(y_0+x)\,\mathbb{E}\Big[ \exp\Big\{-\int_0^\mathbb{T}\gamma(s,y_0+B_{s})\,ds\Big\}\Big|B_\mathbb{T}=x\Big]\,\\
&\hspace*{4cm}\times 
\exp\Big\{\beta(\mathbb{T},y_0+x)-\frac{x^2}{2\mathbb{T}}\Big\} \,dx\\
&=\frac{\sqrt{2\pi \mathbb{T}}}{\Gamma_\mathbb{T}(\mathbb{R})}\,\mathbb{E}\Big[ \psi(y_0+B_\mathbb{T}) \exp\Big\{\beta(\mathbb{T},y_0+B_\mathbb{T})-\int_0^\mathbb{T}\gamma(s,y_0+B_{s})\,ds\Big\} \Big]\\
&=\frac{\sqrt{2\pi \mathbb{T}}}{\Gamma_\mathbb{T}(\mathbb{R})}\,\mathbb{E}[\psi(y_0+B_\mathbb{T})\cdot \hat{M}_\mathbb{T}],
\end{align*}
where $(M_t)_{t\ge 0}$ is the Girsanov martingale defined in \eqref{eq:Girsanov}. We deduce that
\[
\mathbb{E}[\psi(W)]=\frac{\nu(\psi)}{\nu(1)}=\mathbb{E}[\psi(y_0+B_\mathbb{T})\cdot\hat{M}_\mathbb{T}]=\mathbb{E}[\psi(Y_\mathbb{T})],
\]
this equality is satisfied for any non-negative function $\psi$ and therefore corresponds to the announced statement.
\end{proof}
\end{Com}
\subsubsection*{Algorithm introduced by Herrmann and Zucca}

The procedure of the exact simulation was modified by Herrmann and Zucca \cite{herrmann-zucca-exact} to generate the first passage time of continuous diffusion processes. We now focus on the first passage time of the diffusion $Y$, starting in $y_0$, through the level $L$. This algorithm is also based on the Girsanov formula, and requires constructing  a skeleton of a 3-dimensional Bessel process. 

We first recall that in the particular case of Brownian motion, the first passage time through level $L$ denoted by $\tau_L$ satisfies $\tau_L \sim (L-y_0)^2/G^{2}$, where $G \sim \mathcal{N}(0,1)$. The primary idea is therefore to first generate a Brownian crossing and second to accept or reject this variate using the Girsanov weight. The construction of this algorithm looks similar to the algorithms presented by Beskos, Papaspiliopoulos and Roberts. The primary difference is to replace the Brownian paths (or Brownian bridge paths) that appear in the rejection sampling by Bessel paths. The explanation of such a modification derives from the observation that once the Brownian first passage time $\tau_L$ is generated, the Brownian motion constrained to stay under the level $L$ on $[0,\tau_L]$ is related to a $3$-dimensional Bessel process. The algorithm proposal is as follows. 
%
%
%
%
%
%
%
%
%
%
%

\begin{framed}
\centerline{\sc Exact simulation of $\tau_L$ for continuous diffusion -- Algorithm $(H\!Z)$}
\emph{\begin{enumerate}
\item Let $(G_n)_{n\ge 1}$  be independent standard 3-dimensional Gaussian vectors.
\item Let $(e_n)_{n\ge 0}$  be independent exponentially distributed r.v. with an average of $1/\kappa$.
\item Let $(V_n)_{n\geq 1}$ be independent uniformly distributed  r.v. on $[0,1]$. 
\item Let $(g_n)_{n\ge 1}$  be independent standard Gaussian r.variables. 
\end{enumerate}
 The sequences $(G_n)_{n\ge 1}$, $(e_n)_{n\ge 0}$, $(V_n)_{n\geq 1}$ and $(g_n)_{n\ge 1}$ are assumed to be independent.}\\[5pt]
\noindent {\bf Initialisation:} $k=0$, $n=0$.\\[2pt]
{\bf Step 1.} $k \leftarrow k+1$, $\delta = (0,0,0)$, $\mathcal{W} = 0$,  $\mathcal{T}_k \leftarrow (L-y_0)^2/g_k^{2}$, $\mathcal{E}_0 = 0$ and $\mathcal{E}_1 = e_n$.\\[5pt]
{\bf Step 2.}  While $\mathcal{E}_1 \leq \mathcal{T}_k$ do:
\begin{itemize}
\item set  $n\leftarrow n+1$
\item $\displaystyle\delta \leftarrow \frac{\mathcal{T}_k - \mathcal{E}_1}{\mathcal{T}_k- \mathcal{E}_0} \delta+ \sqrt{\frac{(\mathcal{T}_k - \mathcal{E}_1)(\mathcal{T}_k - \mathcal{E}_0)}{\mathcal{T}_k - \mathcal{E}_0}}\ G_n$ 
\item If $\kappa V_n \leq \gamma(\mathcal{E}_1,L - \parallel\mathcal{E}_1(L-y_0)(1,0,0)/\mathcal{T}_k + \delta \parallel)$ then $ \mathcal{W} \leftarrow 1$ else  $\mathcal{W} \leftarrow 0$
\item  $\mathcal{E}_0 \leftarrow \mathcal{E}_1$ and $\mathcal{E}_1 \leftarrow \mathcal{E}_1 + e_n$
\end{itemize}
{\bf Step 3.} If $\mathcal{W} =0$ then $\mathcal{Y}\leftarrow\mathcal{T}_k$ otherwise go to \textit{Step 1.}\\
{\bf Outcome:} the random variable $\mathcal{Y}$. 
\end{framed}
\begin{proposition}[Herrmann-Zucca, 2019] We assume that $\tau_L<\infty$ almost surely where $\tau_L$ is the first passage time of the diffusion \eqref{numero} through the level $L$. Under Assumptions \ref{assum20} and \ref{assum24}, both the outcome $\mathcal{Y}$ of Algorithm $(H\!Z)$ and $\tau_L$ have the same distribution.\label{prop:HZ}
\end{proposition}
The detailed proof of Proposition \ref{prop:HZ} is presented in \cite{herrmann-zucca-exact}. We do not present a sketch of this proof in this study because most of its arguments are  similar to those noted in the proof of Proposition \ref{prop:BR1}. However, Herrmann and Zucca did not investigate a time-dependent drift term as appearing in equation \eqref{numero}; they focused their attention on the homogeneous case. The statement of Proposition \ref{prop:HZ} is therefore an adaptation of their result to the nonhomogeneous case: in this study, the function $\gamma$ depends both on the time and space variables.
\begin{rem} The algorithm $(H\!Z)$ can be adapted to the particular case of a continuous diffusion process starting at time $\mathbb{T}_0>0$ with the value $y$. In this case, $(Y_t)_{t\ge \mathbb{T}_0}$ is the solution of the following stochastic equation:
\[
dY_t=\alpha(t,Y_t)\,dt+dB_t,\quad \forall t\ge \td_0\quad\mbox{and}\quad Y_{\mathbb{T}_0}=y<L.
\]
The definition of the first passage time is modified slightly: $\tau_L$ becomes the first time after $\mathbb{T}_0$ such that the diffusion hits the level $L$. The modifications of the algorithm consist of the substitution of $y_0$ for $y$, $\gamma(\cdot,\cdot)$ for $\gamma(\mathbb{T}_0+\cdot,\cdot)$ and the addition of
$\mathbb{T}_0$ to $\mathcal{Y}$. We then denote  $(H\!Z)_{\mathbb{T}_0}^{y,L}$ the corresponding algorithm.
\label{rem:adapt}
\end{rem}
\subsection{Stopped continuous diffusion}

In the previous section, several procedures of exact simulation were presented:
\begin{itemize}
\item simulation of $Y_\mathbb{T}$: the value of diffusion \eqref{numero} at any fixed time $\mathbb{T}$.
\item simulation of the first passage time through level $L$ for the diffusion, denoted $\tau_L$.
\end{itemize}
We can generate the position, and, we can generate the exit time. To complete the description, we introduce a suitable combination of the time and the position that shall play an essential role later on.
We thus build an algorithm that permits us to obtain the exact simulation of the random couple $(\tau_L\wedge \mathbb{T}, Y_{\tau_L\wedge \mathbb{T}})$ linked to the stopped diffusion. In this study, $(Y_t)_{t\ge 0}$ still stands for continuous diffusion. 

We first introduce a preliminary result about the standard Brownian motion $(B_t)_{t\ge 0}$. One stage of this study is to generate a random variable that has the same conditional distribution as $B_\mathbb{T}$ given $\tau_L>\mathbb{T}$, where $\mathbb{T}$ is fixed and $\tau_L$ is the Brownian first passage time (we shall assume that $L>0$; the other case can be obtained by symmetry arguments). We thus build the following algorithm.

\begin{framed}
\centerline{\sc Conditional Brownian motion given $\tau_L>\mathbb{T}$ -- Algorithm $(C\!B\!M)_{\mathbb{T}}^L$ }
\emph{\begin{enumerate}
\item Let $(G_n)_{n\ge 1}$ be a sequence of independent standard Gaussian random variables
\item Let $(U_n)_{n\geq 1}$ be a sequence of indep. uniformly distributed  random variables on $[0,1]$. 
\end{enumerate}
 The sequences $(G_n)_{n\ge 1}$ and $(U_n)_{n\geq 1}$ are assumed to be independent.}\\[5pt]
\noindent {\bf Initialisation:} $n=1$,  $\mathcal{Y} = 0$.\\[2pt]
{\bf While} $\sqrt{\mathbb{T}}\,G_n > L$ or $-\frac{\mathbb{T}}{2 L}\ln(U_n)> L-\sqrt{\mathbb{T}}\,G_n$ {\bf do} $n \leftarrow n+1$.\\ 
{\bf Set} $\mathcal{Y} \leftarrow \sqrt{\mathbb{T}}G_n$.\\
{\bf Outcome:} The random variable $\mathcal{Y}$.
\end{framed}
\begin{proposition} We let $(B_t)_{t\ge 0}$ be a standard Brownian motion. Then, both the outcome $\mathcal{Y}$ of Algorithm $(C\!B\!M)_\mathbb{T}^L$ and $B_\mathbb{T}$ given $\tau_L>\mathbb{T}$ have the same distribution.\label{prop:CBM}
\end{proposition}
\begin{proof} The proof is based on classical acceptance/rejection sampling. We first describe the joint distribution of $B_\mathbb{T}$ and $\tau_L>\mathbb{T}$ (see, for instance, Lerche \cite{lerche})
\begin{equation*}
u(\td,x)\,dx := \mathbb{P}(\tau_{L}>\mathbb{T}, B_\mathbb{T} \in dx)=\left(\frac{1}{\sqrt{\mathbb{T}}}\ f_G\left(\frac{x}{\sqrt{\mathbb{T}}}\right) - \frac{1}{\sqrt{\mathbb{T}}}\ f_G \left(\frac{x-2L}{\sqrt{\mathbb{T}}}\right)\right)\,dx,
\end{equation*}
where $f_G$ is the density of a standard Gaussian variate. We introduce $F_G$ into the corresponding cumulative distribution. Then, the previous expression leads to:
\begin{equation}
f_\mathbb{T}(x)\,dx:=\mathbb{P}(B_\mathbb{T}\in dx \vert \tau_{L} > \mathbb{T}) = \frac{1}{\sqrt{\mathbb{T}}}\frac{f_G(x/\sqrt{\mathbb{T}})-f_G((x-2L)/\sqrt{\mathbb{T}})}{F_G(L/\sqrt{\mathbb{T}}) - F_G(-L/\sqrt{\mathbb{T}})}\,dx.
\end{equation}
The following upper-bound is satisfied:
\begin{equation*}
f_\mathbb{T}(x)\leq c\ \frac{f_G(x/\sqrt{\mathbb{T}})}{\sqrt{\mathbb{T}}F_G(L/\sqrt{\mathbb{T}})}\ 1_{]- \infty, L]}(x) =: c\,g_\mathbb{T}(x)\quad \mbox{with}\quad c = \frac{F_G(L/\sqrt{\mathbb{T}})}{F_G(L/\sqrt{\mathbb{T}}) - F_G(-L/\sqrt{\mathbb{T}})}.
\end{equation*}
$g_\mathbb{T}(\cdot)$ corresponds to a distribution function: a centred Gaussian distribution of variance $\mathbb{T}$ conditioned to stay under the value $L$. In the rejection procedure, $g_\mathbb{T}$ stands for the proposal distribution. Therefore, we generate a random variable $Z$ with distribution $g_\mathbb{T}$. This variate is accepted if $c\,Ug_\mathbb{T}(Z)\le f_\mathbb{T}(Z)$, where $U$ is a uniformly distributed random variable, independent of $Z$.
The condition just mentioned is equivalent to:
\begin{equation*}
U\le  1 - \exp\left(\frac{2L}{\mathbb{T}}(Z-L)\right).
\end{equation*}
If $G$ stands for a standard Gaussian r.v., then the previous condition is equivalent to:
\[
\sqrt{\mathbb{T}}\,G \le L\quad\mbox{and}\quad -\frac{\mathbb{T}}{2 L}\ln(1-U)\le L-\sqrt{\mathbb{T}}\,G,
\]
the acceptance condition appearing in Algorithm $(C\!B\!M)_{\mathbb{T}}^L$.
\end{proof}
First, we described the generation of the conditional Brownian motion in Proposition \ref{prop:CBM}. As explained in the previous section, we relate the distributions of the Brownian paths to the diffusion paths using the classical Girsanov transformation. An interesting application of this transformation is the simulation of a diffusion value at a fixed time $\mathbb{T}$ subject to $\tau_L>\mathbb{T}$. For a general statement, we consider a diffusion process starting at time $\mathbb{T}_0<\mathbb{T}$ with the value $Y_{\mathbb{T}_0}=y$. This corresponds then to the unique strong solution of the equation:
\begin{equation}
dY_t=\alpha(t,Y_t)\,dt+dB_t,\quad \forall t\ge \mathbb{T}_0,\quad\mbox{and}\quad Y_{\mathbb{T}_0}=y<L,
\label{numerobis}
\end{equation}
where $(B_t)_{t\ge 0}$ is a standard one-dimensional Brownian motion. The algorithm becomes the following algorithm (see  Figure \ref{fig1}).
\begin{figure}
\centering
\includegraphics[width=8cm]{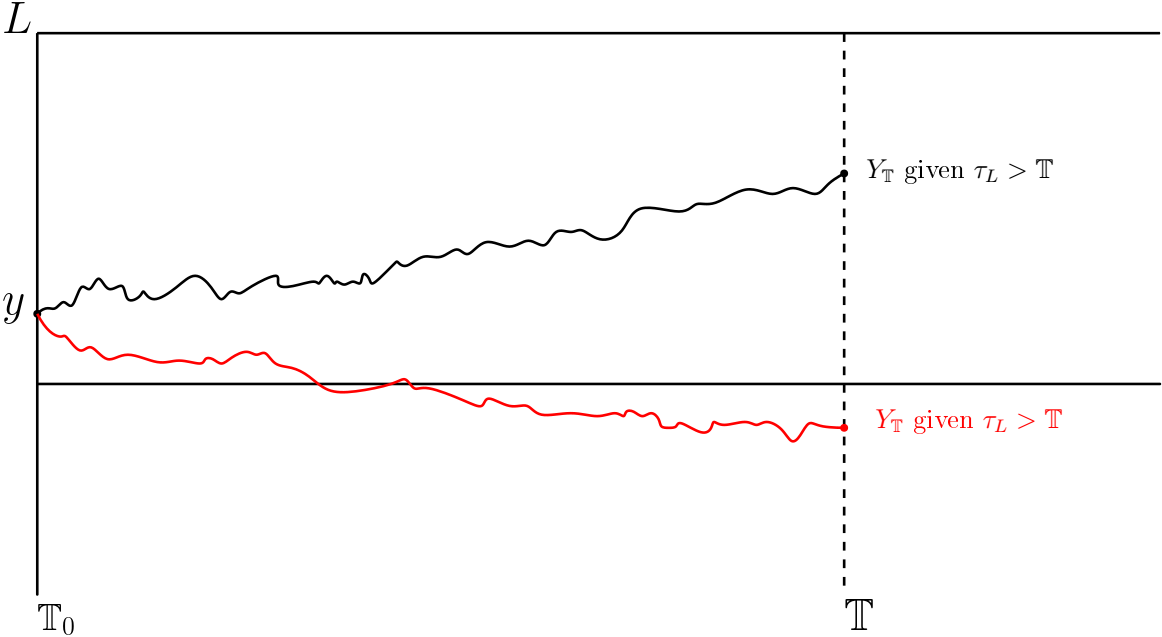}
\caption{Trajectories of the conditioned diffusion}
\label{fig1}
\end{figure}
\begin{framed}
\centerline{\sc Conditioned diffusion $Y_\mathbb{T}$ given $\tau_L>\mathbb{T}$ -- Algorithm $(C\!D)_{\mathbb{T}_0,\mathbb{T}}^{y,L}$ }
\emph{\begin{enumerate}
\item Let $(U_n)_{n\geq 1}$ be independent uniformly distributed  random variables on $[0,\kappa]$. 
\item Let $(E_n)_{n\geq 1}$ be independent exponentially distributed r.v. with an average of $1/\kappa$.
\end{enumerate}
 The sequences $(U_n)_{n\ge 1}$ and $(E_n)_{n\geq 1}$ are assumed to be independent.}\\[5pt]
\noindent {\bf Initialisation:} $n=1$.\\[2pt]
{\bf Step 1.} Set $\mathcal{Y}=y$, $\mathcal{T}=\mathbb{T}_0$ and generate  $\mathcal{Z}\sim (C\!B\!M)_{E_n}^{L-\mathcal{Y}}$.\\[5pt]
{\bf Step 2.} While $\mathcal{T}+E_n<\mathbb{T}$ and $U_n> \gamma(\mathcal{T}+E_n,\mathcal{Y}+\mathcal{Z})$ do
\begin{itemize}
\item $\mathcal{T} \leftarrow \mathcal{T} + E_n$
\item $n\leftarrow n+1$
\item $\mathcal{Y}\leftarrow\mathcal{Y}+\mathcal{Z}$
\item Generate $\mathcal{Z}\leftarrow (C\!B\!M)_{E_n}^{L-\mathcal{Y}}$.
\end{itemize}
{\bf Step 3.}  If $\mathcal{T} +E_n > \td$, then generate $\mathcal{Z}\sim(C\!B\!M)_{\td-\mathcal{T}}^{L-\mathcal{Y}}$, set $\mathcal{Y}\leftarrow \mathcal{Y}+\mathcal{Z}$ and $V \sim \mathcal{U}([0,1])$ independent of all other variates otherwise set $n\leftarrow n+1$ and go to \emph{Step 1.}\\[2pt]
{\bf Step 4.} If $V\cdot\exp(\beta_+) > \exp(\beta(\mathbb{T},\mathcal{Y}))$ then set $n\leftarrow n+1$ and go to \emph{Step 1}.\\[2pt]
{\bf Outcome:}  The random variable $\mathcal{Y}$.
\end{framed}
\begin{proposition} Let us consider $(Y_t)_{t\ge \mathbb{T}_0}$ to be the diffusion defined by \eqref{numerobis} and $\tau_L$ the associated first passage time through level $L$:
\begin{equation}
\tau_L:=\inf\{t\ge \mathbb{T}_0:\ Y_t\ge L\}.
\label{def:tauL2}
\end{equation}
Under Assumptions \ref{assum20}, \ref{assum24} and \ref{assum25}, both the outcome $\mathcal{Y}$ of Algorithm $(C\!D)_{\mathbb{T}_0,\mathbb{T}}^{y,L}$ and $Y_\mathbb{T}$ 
subject to $\tau_L>\mathbb{T}$ have the same distribution.
\label{prop:SCD}
\end{proposition}
\begin{proof} The algorithm $(C\!D)_{\mathbb{T}_0,\mathbb{T}}^{y,L}$ is clearly based on  rejection sampling. The proof therefore uses similar arguments to those pointed out in Proposition \ref{prop:BR1}.  We thus denote the successive values of $\mathcal{Y}$ ($\mathcal{T}$ and $\mathcal{Z}$, respectively) by $\mathcal{Y}_n$ ($\mathcal{T}_n$ and $\mathcal{Z}_n$, respectively). We also introduce a sequence of times $(\mathcal{E}_n)_{n\ge 1}$ defined by $\mathcal{E}_{n+1}=\mathcal{E}_n+E_{n+1}$ with $\mathcal{E}_0=\td_0$. We finally introduce $\mathcal{N}$ the number of Step 1 used until the algorithm stops. 
Because the algorithm is an acceptance/rejection sampling (see Proposition \ref{prop:rejet}),
we have for any nonnegative measurable function $\psi$: 
\begin{align*}
\mathbb{E}[\psi(\mathcal{Y})]&=\frac{\mathbb{E}[\psi(\mathcal{Y})1_{\{ \mathcal{N}=1 \}}]}{\mathbb{P}(\mathcal{N}=1)}=\frac{\nu(\psi)}{\nu(1)}\quad \mbox{where}\quad \nu(\psi):=\mathbb{E}[\psi(\mathcal{Y})1_{\{ \mathcal{N}=1 \}}].
\end{align*}
In the following computations, we denote the particular event $\{ \mathcal{E}_n \leq \mathbb{T} <\mathcal{E}_{n+1}\}$ by $A_n$ and the event by $P_n$:
\[
P_n:=\{ U_1 > \gamma (\mathcal{E}_1,\mathcal{Y}_{1}+\mathcal{Z}_1), \dots , U_n > \gamma (\mathcal{E}_n, \mathcal{Y}_{n}+\mathcal{Z}_n) \},\ \mbox{for}\ n\ge 1\ \mbox{and}\ P_0=\Omega.
\]
We therefore obtain that:
\[
\nu(\psi)=\sum_{n\ge 0} \mathbb{E}\Big[ \psi(\mathcal{Y}_{n+2})1_{P_n}1_{\{V\cdot \exp(\beta_+) < \exp(\beta(\mathbb{T},\mathcal{Y}_{n+2}))\}}1_{A_n} \Big].
\]
Integrating with respect to all uniformly distributed random variables $(U_n)$ and with respect to $V$ leads to
\[
\nu(\psi)=\sum_{n\ge 0} \frac{1}{\kappa^n}\mathbb{E}\Big[\psi(\mathcal{Y}_{n+2}) \prod_{k=1}^n \left( \kappa - \gamma(\mathcal{E}_k,\mathcal{Y}_k+\mathcal{Z}_k)\right)\exp(\beta(\mathbb{T},\mathcal{Y}_{n+2})-\beta_+) 1_{A_n} \Big]
\]
We note that given $A_n$, $(\mathcal{E}_1-\mathbb{T}_0,\ldots,\mathcal{E}_{n}-\mathbb{T}_0)$ has the same distribution as $(V^{(1)},\ldots, V^{(n)})$ an ordered $n$-tuple of uniform random variables $(V_1,\ldots, V_{n})$ on $[0,\mathbb{T}-\mathbb{T}_0]$. The probability of event $A_n$ can be computed using a Poisson distribution of parameter $\kappa (\mathbb{T}-\mathbb{T}_0)$. Finally, on event $A_n$, $(\mathcal{Y}_1, \mathcal{Y}_2,\ldots \mathcal{Y}_{n+2})$ has the same distribution as: \[(y,y+B_{V^{(1)}},\ldots, y+B_{V^{(n)}},y+B_{\mathbb{T}-\mathbb{T}_0})\ \mbox{subject to}\ \tau_{L-y}^B:=\inf\{t\ge 0:\ B_t\ge L-y\}>\mathbb{T}-\mathbb{T}_0\] with $(B_t)_{t\ge 0}$ a standard Brownian motion independent of the $n$-tuple $(V_1,\ldots, V_{n})$. Thus,
\begin{align*}
\nu(\psi)&=\sum_{n\ge 0} \frac{1}{\kappa^n} \mathbb{E}\Big[\psi(y+B_{\mathbb{T}-\mathbb{T}_0})\prod_{k=1}^n \Big(\kappa-\gamma(\mathbb{T}_0+V^{(k)},y+B_{V^{(k)}})\Big) \\
&\qquad\times \exp(\beta(\mathbb{T},y+B_{\mathbb{T}-\mathbb{T}_0})-\beta_+)\Big\vert \tau_{L-y}^B>\mathbb{T}-\mathbb{T}_0\Big]\frac{\kappa^n(\mathbb{T}-\mathbb{T}_0)^{n}}{n!}\,e^{-\kappa (\mathbb{T}-\mathbb{T}_0)}\\
&=\sum_{n\ge 0}\mathbb{E}\Big[\psi(y+B_{\mathbb{T}-\mathbb{T}_0})\prod_{k=1}^n \Big(\kappa-\gamma(\mathbb{T}_0+V_{k},y+B_{V_{k}})\Big) \\
&\qquad\times \exp(\beta(\mathbb{T},y+B_{\mathbb{T}-\mathbb{T}_0})-\beta_+)\Big\vert \tau_{L-y}^B>\mathbb{T}-\mathbb{T}_0\Big]\frac{(\mathbb{T}-\mathbb{T}_0)^{n}}{n!}\,e^{-\kappa (\mathbb{T}-\mathbb{T}_0)}.
\end{align*}
Taking the expectation with respect to the uniformly distributed variates $V_k$ leads to:
\begin{align*}
\nu(\psi)&=\sum_{n\ge 0} \mathbb{E}\Big[\psi(y+B_{\mathbb{T}-\mathbb{T}_0})\Big(\kappa-\frac{1}{\mathbb{T}-\mathbb{T}_0}\int_{0}^{\mathbb{T}-\mathbb{T}_0}\gamma(\mathbb{T}_0+s,y+B_{s})\,ds\Big)^n \\
&\qquad\times \exp(\beta(\mathbb{T},y+B_{\mathbb{T}-\mathbb{T}_0})-\beta_+)\Big\vert \tau_{L-y}^B>\mathbb{T}-\mathbb{T}_0\Big]\frac{(\mathbb{T}-\mathbb{T}_0)^{n}}{n!}\,e^{-\kappa (\mathbb{T}-\mathbb{T}_0)}\\
&=\mathbb{E}\Big[\psi(y+B_{\mathbb{T}-\mathbb{T}_0})\exp\Big(-\int_{0}^{\mathbb{T}-\mathbb{T}_0}\gamma(\mathbb{T}_0+s,y+B_{s})\,ds\Big)\\
&\qquad\times \exp(\beta(\mathbb{T},y+B_{\mathbb{T}-\mathbb{T}_0})-\beta_+) \Big\vert \tau_{L-y}^B>\mathbb{T}-\mathbb{T}_0\Big].
\end{align*}
Because $(y+B_{t})_{t\ge 0}$ subject to $\tau_{L-y}^B>\mathbb{T}-\mathbb{T}_0$ has the same distribution as $(B_t)_{t\ge \mathbb{T}_0}$ subject at once to $B_{\mathbb{T}_0}=y$ and $\tau_L^B\circ\theta_{\mathbb{T}_0}>\mathbb{T}$, where $\theta$ stands for the translation operator, we obtain:
\begin{align*}
\nu(\psi)
&=\mathbb{E}\Big[\psi(B_{\mathbb{T}})\exp\Big(-\int_{\mathbb{T}_0}^{\mathbb{T}}\gamma(s,B_{s})\,ds+\beta(\mathbb{T},B_{\mathbb{T}})-\beta_+\Big) \Big\vert B_{\mathbb{T}_0}=y,\ \tau_L^B\circ\theta_{\mathbb{T}_0}>\mathbb{T}\Big].
\end{align*}
We now modify the expression under review. We introduce: 
\[
\hat{\nu}(\psi):=\mathbb{E}\Big[\psi(B_{\mathbb{T}})1_{\{B_t<L,\ \forall t\in [\mathbb{T}_0,\mathbb{T}]\}}\exp\Big(-\int_{\mathbb{T}_0}^{\mathbb{T}}\gamma(s,B_{s})\,ds+\beta(\mathbb{T},B_{\mathbb{T}})-\beta(\mathbb{T}_0,y)\Big) \Big\vert B_{\mathbb{T}_0}=y\Big].
\]
Also:
\begin{equation}
\mathbb{E}[\psi(\mathcal{Y})]=\frac{\nu(\psi)}{\nu(1)}=\frac{\hat{\nu}(\psi)}{\hat{\nu}(1)}.\label{eq:rapp}
\end{equation}
Because $(M_t)_{t\ge 0}$ is defined by:
\[
M_t:=\exp\Big(-\int_{\mathbb{T}_0}^{\mathbb{T}_0+t}\gamma(s,B_{s})\,ds+\beta(\mathbb{T}_0+t,B_{\mathbb{T}_0+t})-\beta(\mathbb{T}_0,y)\Big),
\]
and is the exponential martingale appearing in the Girsanov transformation, we obtain by change of measure:
\[
\hat{\nu}(\psi)=\mathbb{E}\Big[\psi(Y_{\mathbb{T}})1_{\{Y_t<L,\ \forall t\in [\mathbb{T}_0,\mathbb{T}]\}}\Big],
\]
where $(Y_t)_{t\ge \mathbb{T}_0}$ stands for diffusion \eqref{numerobis}. The ratio \eqref{eq:rapp} permits us to conclude the proof:
\[
\mathbb{E}[\psi(\mathcal{Y})]=\mathbb{E}[\psi(Y_{\mathbb{T}})\vert \tau_L>\mathbb{T} ],
\]
which is the stopping time $\tau_L$ being introduced in the statement \eqref{def:tauL2}.

\end{proof}
Finally, we can write an algorithm that exactly generates  the distribution of the couple $(\tau_L\wedge \mathbb{T}, Y_{\tau_L\wedge \mathbb{T}})$, where $(Y_t)_{t\ge \mathbb{T}_0}$ stands for the continuous diffusion defined in \eqref{numerobis} and $\tau_L$, the first passage time defined in \eqref{def:tauL2} (see Figure \ref{fig2}).
\begin{figure}
\centering
\includegraphics[width=8cm]{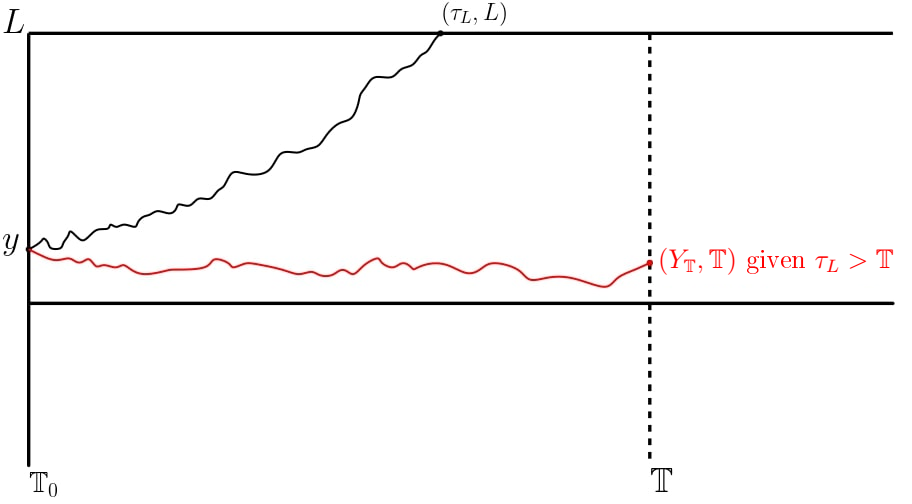}
\caption{Two typical trajectories of continuous diffusion.}
\label{fig2}
\end{figure}

\begin{framed}
\centerline{\sc Stopped diffusion: $(\tau_L\wedge \mathbb{T}, Y_{\tau_L\wedge \mathbb{T}})$  -- Algorithm $(S\!D)_{\mathbb{T}_0,\mathbb{T}}^{y,L}$ }

\vspace*{0.5cm}
\noindent {\bf Step 1.} Generate $\mathcal{T}\sim  (H\!Z)_{\mathbb{T}_0}^{y,L}$ (defined in Remark \ref{rem:adapt}).\\[5pt]
{\bf Step 2.} If $\mathcal{T}<\mathbb{T}$ then set $\mathcal{Y}\leftarrow L$ otherwise generate $\mathcal{Y}\sim (C\!D)_{\mathbb{T}_0,\mathbb{T}}^{y,L}$ and set $\mathcal{T}\leftarrow \mathbb{T}$.\\[2pt]
{\bf Outcome:}  The random couple $(\mathcal{T},\mathcal{Y})$.
\end{framed}
\begin{proposition} We consider $(Y_t)_{t\ge \mathbb{T}_0}$ to be the diffusion defined by \eqref{numerobis} and $\tau_L$ to be the associated first passage time \eqref{def:tauL2}. Under Assumptions \ref{assum20}, \ref{assum24} and \ref{assum25}, both the outcome $(\mathcal{T},\mathcal{Y})$ of Algorithm $(S\!D)_{\mathbb{T}_0,\mathbb{T}}^{y,L}$ and $(\tau_L\wedge \mathbb{T}, Y_{\tau_L\wedge \mathbb{T}})$ have the same distribution.
\label{prop:SD}
\end{proposition}
\begin{proof} Either $\tau_L$ is smaller than $\mathbb{T}$ which corresponds to $\tau_L\wedge \mathbb{T}=\tau_L$ and $Y_{\tau_L\wedge \mathbb{T}}=L$ or $\tau_L$ is larger than $\mathbb{T}$; therefore the distribution of $Y_{\tau_L\wedge \mathbb{T}}$ is the conditional distribution of $Y_\mathbb{T}$ subject to $\tau_L>\mathbb{T}$.
\end{proof}
\section{Simulation of the first passage time for stopped jump diffusion}
\label{sec:algostopdiff}
The goal of this section is to generate the first passage time of jump diffusion. This challenging objective was already considered by Giesecke and Smelov \cite{giesecke-smelov} who adapted the localisation approach introduced in \cite{chen-huang} to that particular context. In this study, we propose a different method based on the work of Herrmann and Zucca \cite{herrmann-zucca-exact} and on the generation of Bessel processes. Jump diffusion is characterised by the stochastic differential equation between the jump times \eqref{afterijump} and the jump height described in \eqref{afterijump1}. We already discussed the possibility of reducing the considered model. Thus, we shall first consider the following reduced model. We introduce  $(T_n)_{n\ge 1}$ to be a sequence of jump times, and the time spent between two consecutive jumps is exponentially distributed; therefore, $T_n=\sum_{k=1}^n E_k$, where $(E_k)_{k\ge 1}$ is a sequence of independent exponentially distributed random variables with an average of $1/\lambda$. The initial position of the diffusion is given by 
$Y_0=y_0$ and the jump diffusion under consideration satisfies 
%
%
\begin{equation}
dY_{t} = \alpha(t,Y_{t})\,dt + dB_t, \quad \mbox{for}\ T_n<t<T_{n+1},\quad n\in\mathbb{N},
\label{afterijump3}
\end{equation} 
the jumps modify the trajectories as follows:
\begin{equation}
Y_{T_n}=Y_{T_{n}-}+j(T_n,Y_{T_n-},\xi_n),\quad \forall n\in\mathbb{N}, 
\label{afterijump4}
\end{equation}
where $j : \mathbb{R}_+ \times \mathbb{R}\times\mathcal{E}\rightarrow \mathbb{R}$ denotes the jump function and $(\xi_n)_{n\ge 1}$ stands for a sequence of independent random variables with distribution function $\phi/\lambda$ (also independent of the  Brownian motion $(B_t)_{t\ge 0}$ and independent of the sequence $(T_n)_{n\ge 0}$). We associate the first passage time through level $L$, where $L>y_0$, with the stochastic process defined by \eqref{afterijump3}--\eqref{afterijump4} as follows:
\begin{equation}
\tau_L:=\inf\{t\ge 0:\ Y_t\ge L\}.
\label{def:fptj}
\end{equation}
Because the jump times and the behaviour of the diffusion process between the jump times are independent, we can use the approach developed in the continuous diffusion case to simulate jump diffusions. The critical argument is that $(t,Y_t)_{t\ge 0}$ is a Markovian stochastic process. Thus, the diffusion paths can be constructed in a piecewise Markovian way. We thus present the algorithm for the generation of the stopped first passage time $\tau_L\wedge \td$ where $\td$ stands for a fixed time.
\begin{framed}
\centerline{\sc Stopped Jump diffusion $(\tau_L\wedge \td)$  -- Algorithm $(S\!J\!D)_{\td}^{y,L}$ }

\emph{\begin{enumerate}
\item Let $(E_n)_{n\geq 1}$ be independent exponentially distributed r.v. with average $1/\lambda$.
\item Let $(\xi_n)_{n\geq 1}$ be independent r.v. with distribution function $\phi/\lambda$.
\end{enumerate}
The sequences $(E_n)_{n\geq 1}$ and $(\xi_n)_{n\geq 1}$ are assumed to be independent.}\\[2pt]
\noindent {\bf Initialisation.} $n=0$, $\mathcal{T}_s=0$ (starting time), $\mathcal{T}_f=0$ (final time), $\mathcal{Y}=y$, $\mathcal{Z}=y$.\\[2pt]
{\bf Step 1.} While ($\mathcal{T}_f<\td$ and $\mathcal{Y}<L$ and $\mathcal{Z}<L$)
do
\begin{itemize}
\item $n\leftarrow n+1$
\item $\mathcal{T}_s\leftarrow\mathcal{T}_f$
\item $\mathcal{T}_f\leftarrow\mathcal{T}_f+E_n$
\item Generate $(\mathcal{S},\mathcal{Z})\sim  (S\!D)_{\mathcal{T}_s,\mathcal{T}_f}^{\mathcal{Y},L}$
\item $\mathcal{Y}\leftarrow \mathcal{Z}+j(\mathcal{T}_f,\mathcal{Z},\xi_n)$
\end{itemize}
{\bf Step 2.} 
\begin{itemize}
\item If $\mathcal{S}>\td$ then set $\mathcal{S}\leftarrow\td$
\item If $\mathcal{S}\le \td$ and $\mathcal{Z}< L$ then $\mathcal{S}\leftarrow\mathcal{T}_f$
\end{itemize}
{\bf Outcome:}  The random variable $\mathcal{S}$.
\end{framed}
\begin{thm} We consider $(Y_t)_{t\ge 0}$ to be the jump diffusion defined by \eqref{afterijump3}--\eqref{afterijump4} and $\tau_L$ to be the associated first passage time \eqref{def:fptj}. Under Assumptions \ref{assum20}, \ref{assum24} and \ref{assum25}, both the outcome $\mathcal{S}$ of Algorithm $(S\!J\!D)_{\td}^{y_0,L}$ and $\tau_L\wedge \td$ have the same distribution.
\label{thm:SJD}
\end{thm}
\begin{figure}[h]
\centering
\includegraphics[width=7cm]{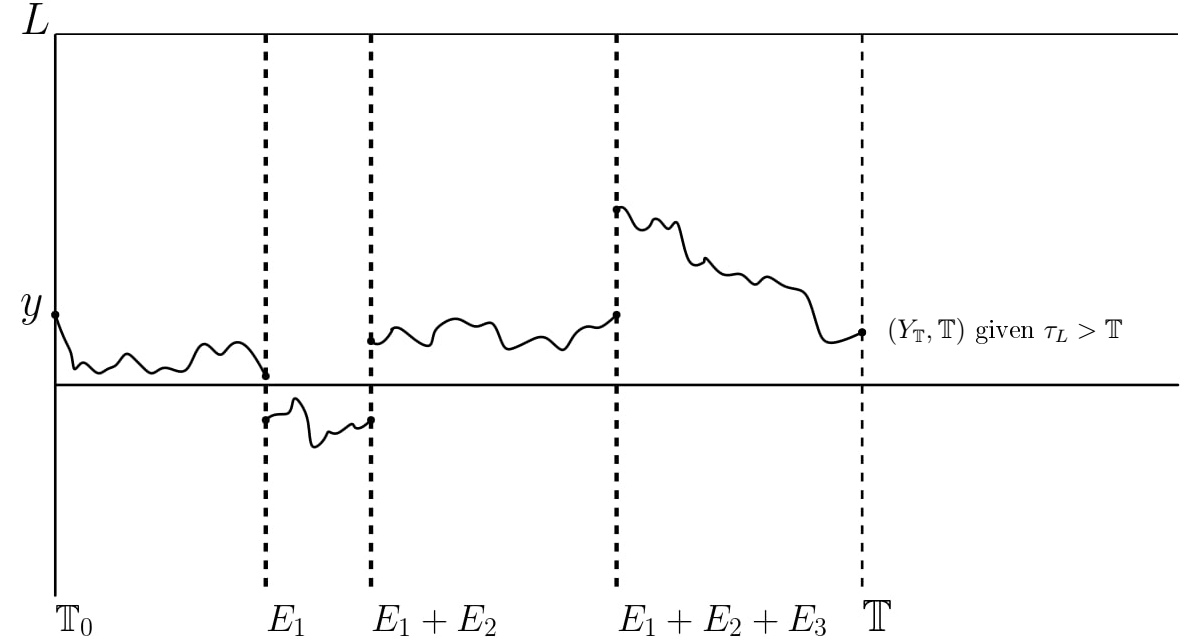}\hspace*{1cm}\includegraphics[width=7cm]{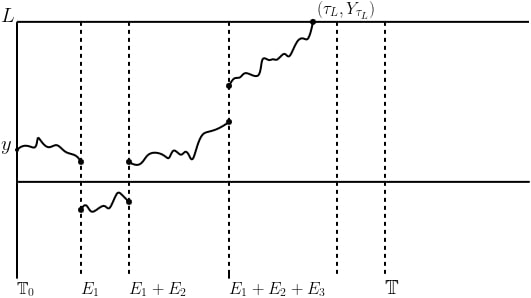}
\label{fig3}
\end{figure}
\begin{figure}[h]
\centering
\includegraphics[width=7cm]{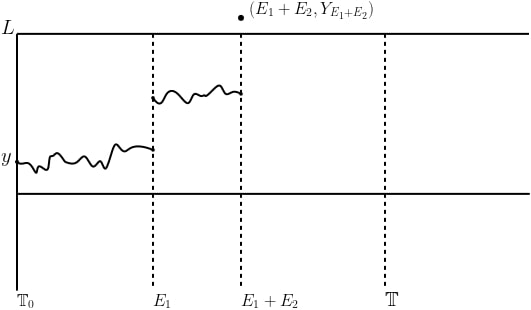}
\caption{Three typical paths representing different scenarios}
\label{fig5}
\end{figure}
\begin{proof} This proof is based on the Markov property of $(t,Y_t)_{t\ge 0}$. We start an iterative procedure. First, we set $\mathcal{T}_s=0$ (starting time) and generate a random variable that represents the first jump time $T_1$. This r.v. is denoted $\mathcal{T}_f=E_1$ in the algorithm (final time). Two situations must therefore be considered (see Figure \ref{fig5}). 
\begin{enumerate}
\item First case: $\mathcal{T}_f>\td$ (i.e., Step 1 only occurs once). Conditional on the event
\emph{"the first jump time is larger than $\td$}", the diffusion behaves like a continuous diffusion process on the interval $[0,\td]$, which is denoted by $(\hat{Y}_t)$. Thus, $Y_t=\hat{Y}_t$ for any $t<\mathcal{T}_f$. The jump times are independent of the Brownian motion driving the diffusive part of the stochastic process. Thus, the results developed in the previous section can be applied. Algorithm $(S\!D)_{\mathcal{T}_s,\mathcal{T}_f}^{y_0,L}$ generates the random couple $(\hat{\tau}_L\wedge \mathcal{T}_f, \hat{Y}_{\hat{\tau}_L\wedge \mathcal{T}_f})$, where $\hat{\tau}_L$ stands for the first passage time of the continuous diffusion $(\hat{Y}_t)$ after time $\mathcal{T}_s$. The random couple is denoted by $(\mathcal{S},\mathcal{Z})$ in Algorithm $(S\!J\!D)_{\td}^{y,L}$. We finally observe different situations:
\begin{itemize}
\item $\hat{\tau}_L\wedge \mathcal{T}_f>\td$, which is equivalent to the Condition $\mathcal{S}>\td$. In such a situation, we can easily deduce that $\hat{\tau}_L\wedge \td=\td=\tau_L\wedge \td$, because $\mathcal{T}_f>\td$. Therefore, we set $\mathcal{S}=\td$ in the algorithm to obtain the distribution identity announced in the statement of Theorem \ref{thm:SJD}.
\item $\hat{\tau}_L\wedge \mathcal{T}_f\le \td$, or  ($\mathcal{S}\le \td$ and $\mathcal{Z}\ge L$). We thus obtain $\tau_L\wedge \td=\tau_L=\tau_L\wedge \mathcal{T}_f=\hat{\tau}_L\wedge \mathcal{T}_f$ which 
is identically distributed as $\mathcal{S}$.
\end{itemize}

\item Second case: $\mathcal{T}_f\le \td$. In this case, the diffusion $(Y_t)$ also corresponds to a continuous diffusion $(\hat{Y}_t)$ on the time interval $[\mathcal{T}_s,\hat{\tau}_L\wedge \mathcal{T}_f[$. Therefore, the previous section explains how to generate $(\hat{\tau}_L\wedge \mathcal{T}_f, \hat{Y}_{\hat{\tau}_L\wedge \mathcal{T}_f})$, denoted by $(\mathcal{S},\mathcal{Z})$. Of course, $\mathcal{S}\le \td$. We distinguish three different situations:
\begin{itemize}
\item $\mathcal{Z}\ge L$ i.e. $\mathcal{S}=\hat{\tau}_L\wedge \mathcal{T}_f=\hat{\tau}_L=\tau_L=\tau_L\wedge \td$, which corresponds to the statement of Theorem \ref{thm:SJD}.
\item $\mathcal{Z}< L$ and $\mathcal{Z}+j(\mathcal{T}_f,\mathcal{Z},\xi_1)\ge L$, where the diffusion does not cross level $L$ until the first jump, while the first jump permits observation of this first passage. The first passage time corresponds to the first jump. Thus, $\tau_L=\mathcal{T}_f$, which occurs in the algorithm in the second step as both $\mathcal{S}\le \td$ and $\mathcal{Z}<L$.
\item $\mathcal{Z}<L$ and $\mathcal{Z}+j(\mathcal{T}_j,\mathcal{Z},\xi_1)<L$. In such a situation, the jump diffusion does not exceed $L$ on the interval $[\mathcal{T}_s,\mathcal{T}_f]=[0,E_1]$.  $\mathcal{Z}<L$ implies that $\hat{\tau}_L\wedge \mathcal{T}_f=\mathcal{T}_f$, the only possibility to overcome the level $L$ on the time interval $[\mathcal{T}_s,\mathcal{T}_f]$ being to observe a suitable jump at time $\mathcal{T}_f$. Unfortunately, such an event cannot occur because  $\mathcal{Z}+j(\mathcal{T}_j,\mathcal{Z},\xi_1)<L$. Thus, the first passage time is strictly larger than the first jump time. To generate the FPT, we propose to start again using the Markov property: we must observe jump diffusion starting at time $\mathcal{T}_f$, which becomes the starting time $\mathcal{T}_s$ with $\mathcal{Z}+j(\mathcal{T}_j,\mathcal{Z},\xi_1)<L$. In such a situation, the algorithm permits repeating Step 1 with new initial values.
\end{itemize}
\end{enumerate}
Because $\tau_L\wedge \td$ is finite a.s., only a finite number of repetitions of Step 1 is observed. The iterative procedure, which is directly associated with the Markov property of jump diffusion, permits us then to obtain the  statement of Theorem \ref{thm:SJD}.
\end{proof}

\begin{Com}
Theorem \ref{thm:SJD} concerns the generation of the finite stopping time $\tau_L\wedge \td$ associated with the reduced model \eqref{afterijump3}--\eqref{afterijump4}. We also recall that $\td$ is a fixed time. Using the Lamperti transformation, it is possible to generalize the study. Let us assume that, between two consecutive jumps, the stochastic process satisfies a stochastic differential equation:
\begin{equation}
dY_{t} = \mu(t,Y_{t})\,dt + \sigma(t,Y_{t})\,dB_t, \quad \mbox{for}\ T_i<t<T_{i+1},\quad i\in\mathbb{N},
\label{afterijumpg}
\end{equation} 
and the jumps modify the trajectories as follows:
\begin{equation}
Y_{T_i}=Y_{T_{i}-}+j(T_i,Y_{T_i-},\xi_i),\quad \forall i\in\mathbb{N}, 
\label{afterijumpg1}
\end{equation}
where $j$ stands for the jump function. A generation of the stopping time $\tau_L\wedge \td$ associated with the jump diffusion \eqref{afterijumpg}--\eqref{afterijumpg1} is then available. Proposition \ref{lamperti} emphasizes the efficient way to generate $\tau_L\wedge\td$.
Let us define
\begin{equation*}
\nu(t,x) = \int_{L}^{x} \frac{1}{\sigma(t,y)}\,dy
\label{lampertibis}
\end{equation*}
and let us consider its inverse $\nu^{-1}: \mathbb{R}_+\times\mathbb{R}\to\mathbb{R}$ which represents the unique function verifying $\nu^{-1}(t,\nu(t,x)) = x$ for any $(t,x) \in \mathbb{R}_+\times \mathbb{R}$. We define $Z_t=\nu(t,Y_t)$. As already mentionned in Section \ref{lamperti}, $(Z_t)_{t\ge 0}$ is a jump diffusion satifying $Z_0=\nu(0,y_0)$ and the reduced model 
\eqref{afterijump3}--\eqref{afterijump4} where the function $\alpha$ corresponds to 
\begin{equation}
\alpha(t,x):=\frac{\partial \nu}{\partial t}(t,\nu^{-1}(t,x)) + \frac{\mu(t,\nu^{-1}(t,x))}{\sigma(t,\nu^{-1}(t,
x))} - \frac{1}{2}\frac{\partial \sigma}{\partial x}(t,\nu^{-1}(t,x))
\label{eq:defdealpha}
\end{equation}
and the jump function $j(t,z,v)$ is replaced by 
\begin{equation}
\hat{\jmath}(t,z,v):=\nu(t,\nu^{-1}(t,z)+j(t,\nu^{-1}(t,z),v))-\nu(t,\nu^{-1}(t,z)).
\label{eq:newjump}
\end{equation}
 Let us finally note that $x=L$ represents the unique solution of the equation $\nu(t,x)=0$ since the diffusion coefficient is strictly positive. Therefore the following identity holds:
\begin{equation}
\tau_L:=\inf\{t\ge 0:\ Y_t\ge L\}=\inf\{t\ge 0:\ Z_t\ge 0\}.
\label{eq:iden:tau}
\end{equation}
\begin{proposition} Let $\td$ be a fixed time. The random variable $\mathcal{S}\sim (S\!J\!D)^{\nu(0,y_0),0}_{\td}$ which is the outcome of the algorithm -- using the drift term $\alpha$ defined in \eqref{eq:defdealpha} and the jump function described in \eqref{eq:newjump} -- has the same distribution as $\tau_L\wedge \td$ the stopping time associated with the jump diffusion \eqref{afterijumpg}--\eqref{afterijumpg1}.
\label{algo-gen-diff}
\end{proposition}
\end{Com}
Theorem \ref{thm:SJD} concerns the generation of the finite stopping time $\tau_L\wedge \td$ associated with the reduced model \eqref{afterijump3}--\eqref{afterijump4}. We also  recall that $\td$ is a fixed time. Using the Lamperti transformation, it is possible to generalise the study (see \cite{these-Nicolas} for more details). 

\begin{Com} 
\section{Simulation of a.s. finite first passage times for jump diffusions}
\label{sec:infinite}
\subsection{Modification of the algorithm}
\label{sec:modif}
In the previous section, we focus our attention on the exact generation of stopped first passage times for jump diffusions, denoted by $\tau_L\wedge \td$ where $\td$ stands for a fixed time. The main advantage of considering stopped processes is to deal with bounded random variables. The algorithms presented so far therefore stop almost surely since they require only a finite number of iterations. 

The aim of this section is to present particular situations where the stopped diffusion can be replaced 
by the diffusion itself. They obviously correspond to almost surely finite first passage times: $\tau_L<\infty$. The algorithm $(S\!J\!D)_{\td}^{y,l}$ can then easily be modified just by setting $\td=\infty$: we obtain the following algorithm $(J\!D)^{y,l}$. 

\begin{framed}
\centerline{\sc Jump diffusion $\tau_L$  -- Algorithm $(J\!D)^{y,L}$ }

\emph{\begin{enumerate}
\item Let $(E_n)_{n\geq 1}$ be independent exponentially distributed r.v. with average $1/\lambda$.
\item Let $(\xi_n)_{n\geq 1}$ be independent r.v. with distribution function $\phi/\lambda$.
\end{enumerate}
The sequences $(E_n)_{n\geq 1}$ and $(\xi_n)_{n\geq 1}$ are assumed to be independent.}\\[2pt]
\noindent {\bf Initialization.} $n=0$, $\mathcal{T}_s=0$ (starting time), $\mathcal{T}_f=0$ (final time), $\mathcal{Y}=y$, $\mathcal{Z}=y$.\\[2pt]
{\bf While} $\mathcal{Y}<L$ and $\mathcal{Z}<L$
{\bf do}
\begin{itemize}
\item $n\leftarrow n+1$
\item $\mathcal{T}_s\leftarrow\mathcal{T}_f$
\item $\mathcal{T}_f\leftarrow\mathcal{T}_f+E_n$
\item Generate $(\mathcal{S},\mathcal{Z})\sim  (S\!D)_{\mathcal{T}_s,\mathcal{T}_f}^{\mathcal{Y},L}$
\item $\mathcal{Y}\leftarrow \mathcal{Z}+j(\mathcal{T}_f,\mathcal{Z},\xi_n)$
\end{itemize}
{\bf Outcome:}  The random variable $\mathcal{S}$.
\end{framed}
The following result is an immediate modification of the statement of Theorem \ref{thm:SJD}.
\begin{corollary} Let us consider $(Y_t)_{t\ge 0}$ the jump diffusion defined by \eqref{afterijump3}--\eqref{afterijump4} and $\tau_L$ the associated first passage time \eqref{def:fptj}.  Under Assumptions \ref{assum20}, \ref{assum24} and \ref{assum25} and assuming $\tau_L<\infty$, both the outcome $\mathcal{S}$ of Algorithm $(J\!D)^{y_0,L}$ and $\tau_L$ have the same distribution.
\label{cor:JD}
\end{corollary}
\end{Com}
%
\mathversion{bold}
\subsubsection*{Models with $\tau_L<\infty$ a.s.}
\mathversion{normal}
If the considered first passage time is a.s. finite, we can  extend the previous algorithm by setting $\td = +\infty$. However, it is not so easy to determine if such a condition is satisfied generally. We now focus on assumptions related to the diffusion characteristics (e.g., diffusion coefficient, drift term, jump measure), which lead to the required event $\tau_L<\infty$. 

In fact, conditions are related to specific tools used for determining the transience or recurrence of stochastic processes. The proofs of this type of result are often based on Lyapunov functions. Wee proposed an interesting study of the recurrence or transience of $d$-dimensional jump diffusion \cite{Wee}. We now present the arguments in the particular one-dimensional case to point out conditions for the first passage time to be a.s. finite.

\begin{Com}
First we shall consider a toy model: a Brownian motion with constant drift term combined with a constant rate  Poisson process. The first passage time of this model has been studied by Kou and Wang \cite{Kou-Wang} with applications to mathematical finance. 
\subsubsection{A toy model: Brownian motion with Poisson jumps (Kou-Wang)}
Kou and Wang consider a particular diffusion satisfying \eqref{afterijumpg}--\eqref{afterijumpg1}, the following double exponential jump process:
\begin{equation}
\label{eq:Kou-Wang}
Y_t=\sigma B_t+\mu t+\sum_{i=1}^{N_t}\xi_i,\quad t\ge 0,\quad Y_0=0, 
\end{equation}
where $(N_t,\,t\ge 0)$ corresponds the Poisson counting process of parameter $\lambda$ associated with the random time sequence $(T_i)_{i\ge 1}$ and $(\xi_i)_{i\ge 1}$ are i.i.d random variables, independent of $(N_t,\, t\ge 0)$. The probability distribution function of $\xi_i$, denoted by $f_\xi$, is the double exponential one (with parameters $\eta_1$ and $\eta_2$):
\begin{equation}
\label{eq:dens}
f_\xi(x)=p \eta_1e^{-\eta_1 x} 1_{\{x\ge 0\}}+q \eta_2 e^{\eta_2 x}1_{\{x< 0\}},\quad\mbox{with}\quad p+q=1.
\end{equation}
In this particular situation, it is possible to compute quite easily the infinitesimal generator of the diffusion: for any twice continuously differentiable function $u$, we have
\begin{equation}
\label{eq:gene}
\mathcal{L}u(y)=\frac{1}{2} \sigma^2 u''(y)+\mu u'(y)+\lambda\int_\mathbb{R} \{u(y+x)-u(y)\} f_\xi(x)\,dx,\quad \forall y\in\mathbb{R}.
\end{equation}
Some information concerning the first passage time can be obtained using this generator. Kou and Wang emphasized for instance the joint distribution of the first passage time $\tau_L$ (for $L>0$) and the position $Y_{\tau_L}$ (useful for the computation of the overshoot $Y_{\tau_L}-L$). The results are deeply based on the fact that the expression of the generator is simple enough for pointing out explicit solutions of different boundary value problems.

First let us consider the Laplace transform of $\tau_L$ that is $\mathbb{E}[e^{-\rho\tau_L}]$ for any $\rho>0$ (see Theorem 3.1 in \cite{Kou-Wang}). There exist two positive roots $\beta_{1,\rho}$ and $\beta_{2,\rho}$ satisfying 
\[
\rho=G(\beta)\quad \mbox{with}\quad G(y):=y\mu+\frac{1}{2}\,y^2\sigma^2+\lambda\Big( \frac{p\eta_1}{\eta_1-y}+\frac{q\eta_2}{\eta_2+y}-1 \Big).
\]
We define two constants: 
\[
A(\rho)=\frac{\eta_1-\beta_{1,\rho}}{\eta_1}\frac{\beta_{2,\rho}}{\beta_{2,\rho}-\beta_{1,\rho}}\quad \mbox{and}\quad B(\rho)=\frac{\beta_{2,\rho}-\eta_1}{\eta_1}\frac{\beta_{1,\rho}}{\beta_{2,\rho}-\beta_{1,\rho}}
\]
and introduce the following continuous function $u_\rho$:
\begin{equation}
\label{eq:defsolLap}
u_\rho(y)=1_{\{y\ge L\}}+\Big( A(\rho)\,e^{-\beta_{1,\rho}(L-y)}+ B(\rho)\,e^{-\beta_{2,\rho}(L-y)}\Big)1_{\{ y<L \}}.
\end{equation}
It is straightforward that the particular choice of the constants leads $u_\rho$ to satisfy the equation 
\[
-\rho u_\rho(y)+\mathcal{L}u_\rho(y)=0\quad\mbox{for all}\ y<L.
\]
Using some regularization technique (replacing the continuous function $u_\rho$ by a sequence of $\mathcal{C}^2$-continuous functions $u_\rho^{(n)}$) and the martingale theory, Kou and Wang proved that
\[
u_\rho(0)=\mathbb{E}[e^{-\rho \tau_L}u_\rho(Y_{\tau_L})1_{\{\tau_L<\infty\}}]=\mathbb{E}[e^{-\rho \tau_L}],
\]
since $u_\rho(y)=1$ for any $y\ge L$. In order to describe the probability that $\tau_L$ is finite, it suffices to consider $\lim_{\rho\to 0}u_\rho(0)$. Kou and Wang studied carefully the behaviour of all parameters depending on $\rho$ as $\rho\to 0$ and deduced the following result.
\begin{proposition}[Kou \& Wang, 2003]
The first passage time $\tau_L$ of the diffusion \eqref{eq:Kou-Wang} satisfies
\[
\mathbb{P}(\tau_L<\infty)=1\quad\mbox{iff}\quad \overline{u}:=\mu+\lambda\Big( \frac{p}{\eta_1}-\frac{q}{\eta_2} \Big)\ge 0,
\]
where $\overline{u}$ stands for the overall drift of the jump diffusion process.
\end{proposition}
Let us just note that the condition $\overline{u}$ seems quite intuitive and can be related to the asymptotic behaviour of the diffusion paths. Indeed let us consider the sequence $Z_n:=Y_{T_{n}}-Y_{T_{n-1}}$ for $n\ge 2$ and $Z_1=Y_{T_1}$. We observe that $(Z_n)_{n\ge 1}$ is a sequence of i.i.d random variables with a finite second moment. Therefore the law of large numbers implies
\[
\frac{1}{n}\,Y_{T_n}=\frac{1}{n}\sum_{i=1}^n Z_i\to \mathbb{E}[Z_1]\quad \mbox{a.s. \ when}\ n\to \infty.
\]
Since \(\mathbb{E}[Z_1]=\overline{u}/\lambda \), the condition $\overline{u}>0$ obviously leads to the almost sure event $\tau_L<\infty$. For the particular case $\overline{u}=0$, an argument based on the functional central limit theorem (Donsker's theorem) permits also to reach the same conclusion. These arguments are linked to the following facts:
\begin{itemize}
\item the process has independent increments 
\item the increments $Y_t-Y_s$ and $Y_{t-s}-Y_0$ are identically distributed.
\end{itemize}
This restrictive model permits to handle with a first example of jump diffusion with alsmost surely finite first passage time $\tau_L$.
\end{Com}
We first recall the definition of jump diffusion \eqref{jdeq} with time-homogeneous coefficients:
\begin{equation}
dX_t =  \mu(X_{t-}) \,dt + \sigma(X_{t-}) \,dB_t + \int_{\mathcal{E}}j(X_{t-},v)p_{\phi}(dv \times dt),\quad t\ge 0,\quad X_0=y_0,
\label{jdeqhom}
\end{equation}
where $p_{\phi}(dv \times dt)$ is a Poisson measure of intensity $\phi(dv)dt$. We introduce the compensated Poisson measure associated with $p_\phi$ and defined by:
\[
\hat{p}_\phi(dv\times dt):=p_\phi(dv\times dt)-\phi(dv)dt.
\]
Equation \eqref{jdeqhom} can easily be rewritten using $\hat{p}_\phi$ instead of $p_\phi$ just by replacing the drift term of the diffusion process $\mu(\cdot)$ by:
\[
\hat{\mu}(x)=\mu(x)+\int_{\mathcal{E}}j(x,v)\phi(dv),\quad \forall x\in\mathbb{R}.
\]
$(X_t,\, t\ge 0)$ has the same conditional distribution as $(Y_t,\ t\ge 0)$ defined by \eqref{afterijump}--\eqref{afterijump1} with time-homogeneous coefficients. Therefore, we focus on a first passage time problem for diffusion $(X_t,\, t\ge 0)$.
In the particular case $\mathcal{E}=\mathbb{R}$, we obtain the jump diffusion model introduced by Wee:
\begin{equation}
\label{eq:depWee}
X_t=y_0+\int_0^t\hat{\mu}(X_{s-})\,ds+\int_0^t\sigma(X_{s-})dB_s+\int_0^t\int_{\mathbb{R}}j(X_{s-},v)\hat{p}_\phi(dv\times ds),\quad t\ge 0.
\end{equation}
We denote the probability distribution of such a solution by $\mathbb{P}_{y_0}$. To state the result concerning the first passage time $\mathbb{P}_{y_0}(\tau_L<\infty)=1$ for $y_0<L$, we must introduce several assumptions; the symmetric case can be handled using similar arguments. We assume in particular that the coefficients are regular and the diffusion is nondegenerate.
\begin{assu}
\label{assu:Wee1} We assume that there exists a constant $K>0$ such that:
\begin{equation}
\vert \hat{\mu}(x) - \hat{\mu}(y) \vert^2 + \vert \sigma(x) - \sigma(y) \vert^2 + \int_\mathbb{R} \vert j(x,v) - j(y,v) \vert^2 \phi(dv) \leq K \vert x - y \vert^2
\label{exist1}
\end{equation}
and: 
\begin{equation}
\vert \hat{\mu}(x)  \vert^2 + \vert \sigma(x)  \vert^2 + \int_\mathbb{R} \vert j(x,v)  \vert^2 \phi(dv) \leq K(1+ \vert x \vert^2),
\label{exist2}
\end{equation}
there exists $\sigma_0>0$ such that $\sigma(x)\ge \sigma_0$ for all $x\in\mathbb{R}$.
\end{assu}
Under Assumption \ref{assu:Wee1}, \eqref{eq:depWee} yields a unique strong solution that is right-continuous with left-hand limits (e.g. see Theorem 9.1 in \cite{ikeda-watanabe} for homogeneous coefficients and Theorem 1.19 in \cite{oksendal2005applied} in the general nonhomogeneous case). We now introduce two additional assumptions that are critical to obtain nearly certain finite times $\tau_L$ when the initial value $y_0$ satisfies $y_0<L$.
\begin{Com}
In order to state the main result of this section, we need to first prove that the jump diffusion exits from any bounded interval almost surely.
\begin{lemma}
\label{taufini}
For any positive $R$, let us define the following stopping time associated with the diffusion \eqref{eq:depWee}:
\begin{equation}
\label{def:zeta}
\zeta_{R} = \inf\{t \geq 0:\ X_t \notin ]-R,R[ \}.
\end{equation}
Under Assumption \ref{assu:Wee1}, for any $R>0$ there exists $\rho_R>0$ such that
\begin{equation}
\sup\limits_{y_0 \in [-R,R]} \mathbb{E}_{y_0}[\exp(\rho_R\, \zeta_{R})] < \infty.
\end{equation}
\end{lemma}
\end{Com}
\begin{Com}
\noindent We just recall the proof proposed by Wee.
\begin{proof}
Let us consider $R>0$, choose $a > 3R$ and set $K = a^{2n}$ where $n$ is a positive integer that shall be determined later on. 
We consider $\psi \in C_c^2(\mathbb{R})$ such that 
\begin{equation*}
\psi(z)= \left\{
    \begin{array}{ll}
         K-z^{2n} & \mbox{ for } \vert z \vert \leq a,\\
         0 & \mbox{ for } \vert z \vert \geq a+1.
\end{array}
\right.
\end{equation*}
Let us set $\widehat{\psi}(z) = \psi(z - 2R)$ and $\Psi(t,z) = e^{\rho t } \widehat{\psi}(z)$ for some $\rho >0$ which shall once again be selected later. Using Itô's formula to the diffusion process \eqref{eq:depWee}, we have
\begin{equation}
\mathbb{E}_{y_0}[\Psi(t \wedge \zeta_{R}, X_{t \wedge \zeta_{R}})] = \widehat{\psi}(x) + \mathbb{E}_{y_0}\left[\int_0^{t \wedge \zeta_{R}}\Big(\rho e^{\rho s} \widehat{\psi}(X_s) + e^{\rho s} \mathcal{L}\widehat{\psi}(X_s)\Big)\, ds \right],
\label{ito}
\end{equation}
where
\begin{equation}
\label{def:gene}
\mathcal{L}f(y) = \frac{\sigma^2(y)}{2}f''(y) + \hat{\mu}(y) f'(y) + \int_\mathbb{R} \Big(f(y + j(y,v)) - f(y) -f'(y)j(y,v)\Big)\, \phi(dv).
\end{equation}
Let us define $I_a=\{v\in\mathbb{R}:\ \vert y + j(y,v) - 2R\vert >a \}$. Therefore, for any $\vert y \vert \leq R$,
\begin{align*}
\mathcal{L}\widehat{\psi}(y) &\leq -2n(y-2R)^{2n-2}\Big[(y-2R)\hat{\mu}(y) + (2n-1)\frac{\sigma^2(y)}{2}
&-(y-2R)\int_{I_a} j(y,v) \phi(dv)\Big].
\end{align*}
The triangle inequality leads to $I_a\subset J_a:=\{v\in\mathbb{R}:\ \vert  j(y,v) \vert >a-3R \}$.
Let us then remark that for $\vert y \vert \leq R$,
\begin{align*}
\int_{I_a} |j(y,v)| \phi(dv) &\leq \int_{J_a} |j(y,v)| \phi(dv)
\leq (a-3R)\int_{J_a} \frac{|j(y,v)|}{a-3R} \phi(dv)\\
&\leq (a-3R)\int_{J_a} \left(\frac{|j(y,v)|}{a-3R}\right)^2 \phi(dv)\leq \frac{K(1+ y^2)}{a-3R}\leq \frac{K(1+R^2)}{a-3R}.
\end{align*}
Hence, for $n$ large enough, there exists $\alpha > 0$ such that 
\(\mathcal{L}\widehat{\phi}(y) \leq - \alpha \mbox{, for } \vert y \vert \leq R.\)
Then for $\rho > 0$ small enough, we get $\beta := \alpha -\rho K>0$. Hence equation \eqref{ito} becomes
\begin{align*}
\mathbb{E}_{y_0}[\Psi(t \wedge \zeta_{R}, X_{t \wedge \zeta_{R}})] &\leq K +\mathbb{E}_{y_0}\Big[ \int_0^{t \wedge \zeta_{R}} (\rho K - \alpha)e^{\rho s}\,ds\Big]= K - \beta \, \mathbb{E}_{y_0}\Big[\int_0^{t \wedge \zeta_{R}} e^{\rho s}\,ds\Big].
\end{align*}
Finally, since $\mathbb{E}_{y_0}[\Psi(t \wedge \zeta_{R}, X_{t \wedge \zeta_{R}})] \geq 0$, we obtain
\begin{equation*}
\mathbb{E}_{y_0}[e^{\rho\, t \wedge \zeta_{R}}] \leq \frac{\rho}{\beta} K + 1,
\end{equation*}
which leads to the announced result as $t$ tends to infinity.
\end{proof}
\end{Com}
\begin{assu}
\label{assu:Wee2bis} There exists $r>0$ such that:
\begin{itemize}
\item the following bound holds:
\begin{equation}
\sup\limits_{y \le -r}\int_\mathbb{R} \left( \ln\left(\frac{\vert y + j(y+L+r,v) \vert}{\vert y \vert}\right)\right)^2\phi(dv)=:\kappa_{L,r}< \infty.\label{eq:assu:Wee21}
\end{equation}
\item there exist  $\epsilon>0$ and $\eta>0$ satisfying:
\begin{align}
y\hat{\mu}(y+L+r) + \int_\mathbb{R}\left(y^2\ln\left(\frac{\vert y + j(y+L+r,v)\vert}{\vert y \vert}\right) - yj(y+L+r,v)\right) \phi(dv) \nonumber\\
< \frac{(1-\epsilon)}{2}\,\sigma^2(y+L+r) - \eta y^2,
\label{condition}
\end{align}
for all $y\le - r$.
\end{itemize}
\end{assu}
%
The first part of Assumption \ref{assu:Wee2bis} requires that the jumps are not too large. The second part highlights  a type of competition between the drift, the diffusion coefficient and the jump measure. A careful reading of the proof of Theorem 1 in \cite{Wee} permits us to adapt its statement to the situation just introduced above.
\begin{thm} 
Let Assumption \ref{assu:Wee1} and Assumption \ref{assu:Wee2bis} be satisfied; then, $\mathbb{P}_{y_0}(\tau_{L}< \infty) = 1$ for any $y_0\le L$, where $\tau_L$ is the first passage time through level $L$ for diffusion \eqref{eq:depWee}.
\label{thm:Wee}
\end{thm}
We have just completed this section by considering a particular situation that insures that $\tau_L<\infty$ is nearly certain. This result concerns a time-homogeneous jump diffusion process because the essential tools used in the proof of Theorem \ref{thm:Wee} are based on the infinitesimal generator and on an explicit Lyapunov function. However, the use of suitable comparison results 
permits us to manage nonhomogeneous jump diffusions and to bring up particular conditions that ensure that the first passage times are almost certainly finite. 
\begin{Com}
\begin{proof}Let us first just mention that the particular case $y_0=L$ is obvious. So let us assume for the sequel that $y_0<L$. Let us consider the constant $r>0$ appearing in Assumption \ref{assu:Wee2bis} and introduce a parameter $\delta$ such that $0< \delta < \epsilon(r)/2$, the value of $\delta$ shall be determined later. We also introduce the first entrance time
\begin{equation}
\label{eq:def:zeta2}
\zeta_{L,r} = \inf\{t \geq 0:\ X_t \in ]L,L+2r[ \}.
\end{equation}
Since the paths of the diffusion are not continuous, the first passage time $\tau_L=\inf\{t\ge 0:\ X_t\ge L\}$ does not always correspond to the first time the diffusion reaches the level $L$. That is why we need to give in some sense more thickness to the level $L$: we replace it by the strip $]L,L+2r[$.
Let us prove now that $\zeta_{L,r}$ is almost surely finite which implies $\tau_L<\infty$.

We also introduce a non positive and non increasing symmetric function $F \in C^2(\mathbb{R})$ such that $F(y)  = -\vert y \vert^{2\delta}$ for any $\vert y \vert > \alpha$ where $\alpha=re^{-1/\delta}$. Let us define $f$ by $f(y) = F(\vert y-(L+r) \vert)$ and $I_{y,\alpha}:=\{v\in\mathbb{R}: \vert y + j(y,v) -(L+r) \vert > \alpha\}$. Then, for $\vert y-(L+r) \vert>r$, we obtain 
\begin{align*}
\mathcal{L}f(y) &= 2\delta \vert y - (L+r) \vert^{2 \delta} \Big[\frac{1 - 2\delta}{2 \, \vert y-(L+r) \vert^2}\,\sigma^2(y) -\frac{y-(L+r)}{\vert y - (L+r) \vert^2 }\,\hat{\mu}(y)\\
&+ \int_{I_{y,\alpha}^c} \left( \frac{f(y + j(y,v))}{2\delta \vert y-(L+r) \vert^{2\delta}} + \frac{1}{2\delta} + \frac{y-(L+r)}{\vert y-(L+r) \vert^2}\,j(y,v)\right) \phi(dv)\\
&-\frac{1}{2\delta} \int_{I_{y,\alpha}} \left(\Big(\frac{\vert y + j(y,v) -(L+r) \vert }{\vert y-(L+r) \vert}\Big)^{2 \delta} - 1 -2\delta\, \frac{y-(L+r)}{\vert y-(L+r) \vert^2}\,j(y,v)\right) \phi(dv)\Big]
\end{align*}
Here $\mathcal{L}$ is the infinitesimal generator defined in \eqref{def:gene}.
Then, using the second condition \eqref{condition} in Assumption \ref{assu:Wee2bis}, we obtain
\begin{align*}
\mathcal{L}f(y) \geq 2\delta \vert y-(L+r)\vert^{2\delta} \Big[ \frac{(\epsilon - 2\delta)\sigma^2(y)}{2 \vert y-(L+r) \vert^2} + \eta +\int_{I_{y,\alpha}^c} \ln\left(\frac{\vert y + j(y,v) -(L+r) \vert }{\vert y-(L+r) \vert}\right) \phi(dv) \\
-\frac{1}{2\delta} \int_{I_{y,\alpha}} \left(\Big(\frac{\vert y + j(y,v) -(L+r) \vert }{\vert y-(L+r) \vert}\Big)^{2 \delta} - 1 -2\delta \ln\left(\frac{\vert y + j(y,v) -(L+r) \vert }{\vert y-(L+r) \vert}\right) \right) \phi(dv)\Big].
\end{align*}
Let $M>1$. The set $I_{y,\alpha}$ can be spitted into two parts $I_{y,\alpha}=J_1^M \cup J_2^M$ where 
\begin{eqnarray*}
\left\{\begin{array}{l} J_1^M=\{v\in\mathbb{R}:\ \vert y + j(y,v) -(L+r) \vert > M \vert y-(L+r) \vert\}\\[3pt] J_2^M=\{v\in\mathbb{R}:\ \alpha<\vert y + j(y,v) -(L+r) \vert \le M \vert y-(L+r) \vert \}. \end{array}\right.
\end{eqnarray*}
Combining the previous decomposition with the following bound
\begin{equation*}
(\ln u)^2\ge 1-\frac{1+\ln u}{u},\quad \forall u>0,
\end{equation*}
permits to write:
\begin{align}
\mathcal{A}:=&\frac{1}{2\delta} \int_{I_{y,\alpha}} \left(\left(\frac{\vert y + j(y,v) -(L+r) \vert }{\vert y-(L+r) \vert}\right)^{2 \delta} - 1 -2\delta \ln\left(\frac{\vert y + j(y,v) -(L+r) \vert }{\vert y-(L+r) \vert}\right) \right) \phi(dv)\nonumber\\
&\leq \delta\int_{J_1^M} \left(\frac{\vert y + j(y,v) -(L+r) \vert }{\vert y-(L+r) \vert}\right)^{2 \delta}\left(\ln\left(\frac{\vert y + j(y,v) -(L+r) \vert }{\vert y-(L+r) \vert}\right)\right)^2\phi(dv)\nonumber\\
&+ \delta M^{2\delta} \int_{J_2^M} \left(\ln\left(\frac{\vert y + j(y,v) -(L+r) \vert }{\vert y-(L+r) \vert}\right)\right)^2\phi(dv).
\label{eqA1}
\end{align}
Assumption \ref{assu:Wee2bis} leads to
\begin{align}
\mathcal{A}&\leq \delta\int_{J_1^M} \frac{\vert y + j(y,v) -(L+r) \vert^2}{\vert y-(L+r) \vert^2}\ \phi(dv) + \delta M^{2\delta} \kappa_{L+r}.
\label{eqA2}
\end{align}
On the set $J_1^M$ we have
\begin{align*}
\vert y + j(y,v) - (L+r) \vert &\leq \vert y-(L+r) \vert + \vert j(y,v) - j(L+r,v) \vert + \vert j(L+r,v) \vert\\
&\leq \frac{1}{M} \vert y + j(y,v) - (L+r) \vert+ \vert j(y,v) - j(L+r,v) \vert + \vert j(L+r,v) \vert.
\end{align*}
Consequently
\begin{equation*}
\vert y + j(y,v) - (L+r) \vert \leq \frac{M}{M-1}\,\Big(\vert j(y,v) - j(L+r,v) \vert + \vert j(L+r,v) \vert\Big).
\end{equation*}
Hence, on the considered set
\begin{align*}
&\int_{J_1^M} \frac{\vert y + j(y,v) -(L+r) \vert^2}{\vert y-(L+r) \vert^2}\, \phi(dv) \\
&\leq \frac{M^2}{(M-1)^2} \frac{2}{\vert y-(L+r) \vert^2}\Big[ \int_\mathbb{R} \vert j(y,v) - j(L+r,v) \vert^2 \phi(dv)+ \int_\mathbb{R} \vert j(L+r,v) \vert^2 \phi(dv)\Big].
\end{align*}
Using Assumption \ref{assu:Wee1}, we obtain
\begin{align}
\label{eq:prop1}
\int_{J_1^M} \frac{\vert y + j(y,v) -(L+r) \vert^2}{\vert y-(L+r) \vert^2} \phi(dv) &\leq \frac{2K M^2}{(M-1)^2} \left(1 + \frac{1 + \vert L+r \vert^2}{\vert y-(L+r) \vert^2}\right) \nonumber\\
&\leq \frac{2KM^2}{r^2(M-1)^2}(1 + \vert L+r\vert^2 + r^2).
\end{align}
We deduce that \eqref{eqA2} becomes
\begin{align*}
\mathcal{A}&\leq \frac{2\delta KM^2}{r^2(M-1)^2}\ (1 + \vert L+r\vert^2 + r^2) + \delta M^{2\delta} \kappa_{L+r}.
\end{align*}
Finally, there exists some finite constant $C_{L,r}>0$ such that
\begin{equation*}
\mathcal{L}f(y) \geq 2\delta \vert y-(L+r)\vert^{2\delta} \left[ \frac{(\epsilon - 2\delta)\sigma^2(y)}{2 \vert y-(L+r) \vert^2}+ \eta - \delta\, C_{L,r}\right].
\end{equation*}
Choosing $\delta$ small enough leads to $\mathcal{L}f(y) \geq 0$ for any $y\le L$. Under these conditions, the selected function $f$ is a Lyapunov function. This leads to 
\begin{equation*}
f(y_0) \leq \mathbb{E}_{y_0}[f(X_{\zeta_{L,r}}) 1_{\zeta_{L,r}< \zeta_{R}}] + \mathbb{E}_{y_0}[f(X_{\zeta_R}) 1_{\zeta_{R}< \zeta_{L,r}}],
\end{equation*}
where $\zeta_R$, respectively $\zeta_{L,r}$, is defined by \eqref{def:zeta}, resp. \eqref{eq:def:zeta2}.
Since $f$ is a non positive function and since \[
\mathbb{E}_{y_0}[f(X_{\zeta_{R}}) 1_{\zeta_{R}< \zeta_{L,r}}] \leq F(R-|L|-r) \mathbb{P}_{y_0}(\zeta_{R}< \zeta_{L,r})\le -(R-|L|-r)^{2\delta} \mathbb{P}_{y_0}(\zeta_{R}< \zeta_{L,r}),\] for $R$ large enough, the previous equation becomes
\begin{equation*}
f(y_0) \leq -(R-|L|-r)^{2\delta}\, \mathbb{P}_{y_0}(\zeta_{R}< \zeta_{L,r}).
\end{equation*}
Letting $R$ tend to infinity in the previous inequality implies that $\lim_{R\to\infty}\mathbb{P}_{y_0}(\zeta_{R}< \zeta_{L,r})=0$. Let us suppose that $E=\{\zeta_{L,r} = \infty\}$ is an event with positive probability $\mathbb{P}(E)>0$. Then there exists $R_0>0$ large enough such that 
\begin{equation*}
\mathbb{P}_{y_0}(\zeta_{R_0}=\infty)\ge \mathbb{P}_{y_0}(E \cap \{\zeta_{R_0}\ge \zeta_{L,r}\}) >0.
\end{equation*}
This inequality contradicts the statement of Lemma \ref{taufini}. We deduce therefore by a reductio ad absurdum argument that $\mathbb{P}_{y_0}(\zeta_{L,r}< \infty) =1$.
\end{proof}
\end{Com}
\section{Numerical illustrations}
\label{sec:infinite}
In this last section, two examples of jump diffusion processes were presented, and we proposed simulation experiments that show the methodology introduced in Section \ref{sec:algostopdiff}. The diffusion $(Y_t,\, t\ge 0)$ satisfies \eqref{afterijump}-\eqref{afterijump1} with a deterministic starting value $Y_0=y_0$. In all examples and without a loss of generality, we focus on the case $y_0<L$, where $L$ is the level the diffusion process should overcome, and $\tau_L$ is the first passage time. 
\subsubsection*{Example of stopped jump diffusion.}
We consider the time-homogeneous jump diffusion \eqref{afterijump}-\eqref{afterijump1} with coefficients $\mu(t,y)=2+\sin(y)$ and $\sigma(t,y)=1$. The diffusion starting in $y_0$  satisfies the following SDE between the jump times:
\begin{equation}
dY_t = (2+\sin(Y_t))\,dt + dB_t, \quad \mbox{for}\ T_i<t<T_{i+1},\quad i\in\mathbb{N},
\label{exsin}
\end{equation}
and the jumps satisfy:
\[
Y_{T_i}=Y_{T_i-}+j(T_i,Y_{T_i-},\xi_i),\quad i\in\mathbb{N}.
\]
$T_i=\sum_{k=1}^i E_k$, where $(E_k)_{k\ge 1}$ is a sequence of exponentially distributed random variables with an average of $1/\lambda$. We set $\lambda=1$ throughout the section. $(\xi_i)_{i\ge 1}$ is a sequence of i.i.d. variates with density $\phi$, specified later on. We fix $\mathbb{T}>0$ and thus aim to generate $\tau_L\wedge \mathbb{T}$ using Algorithm $(S\!J\!D)_{\td}^{y,L}$ (see Theorem \ref{thm:SJD}). Therefore, we must verify that Assumptions \ref{assum20}, \ref{assum24} and \ref{assum25} hold. 

\begin{figure}[h]
   \centerline{\includegraphics[width=9cm]{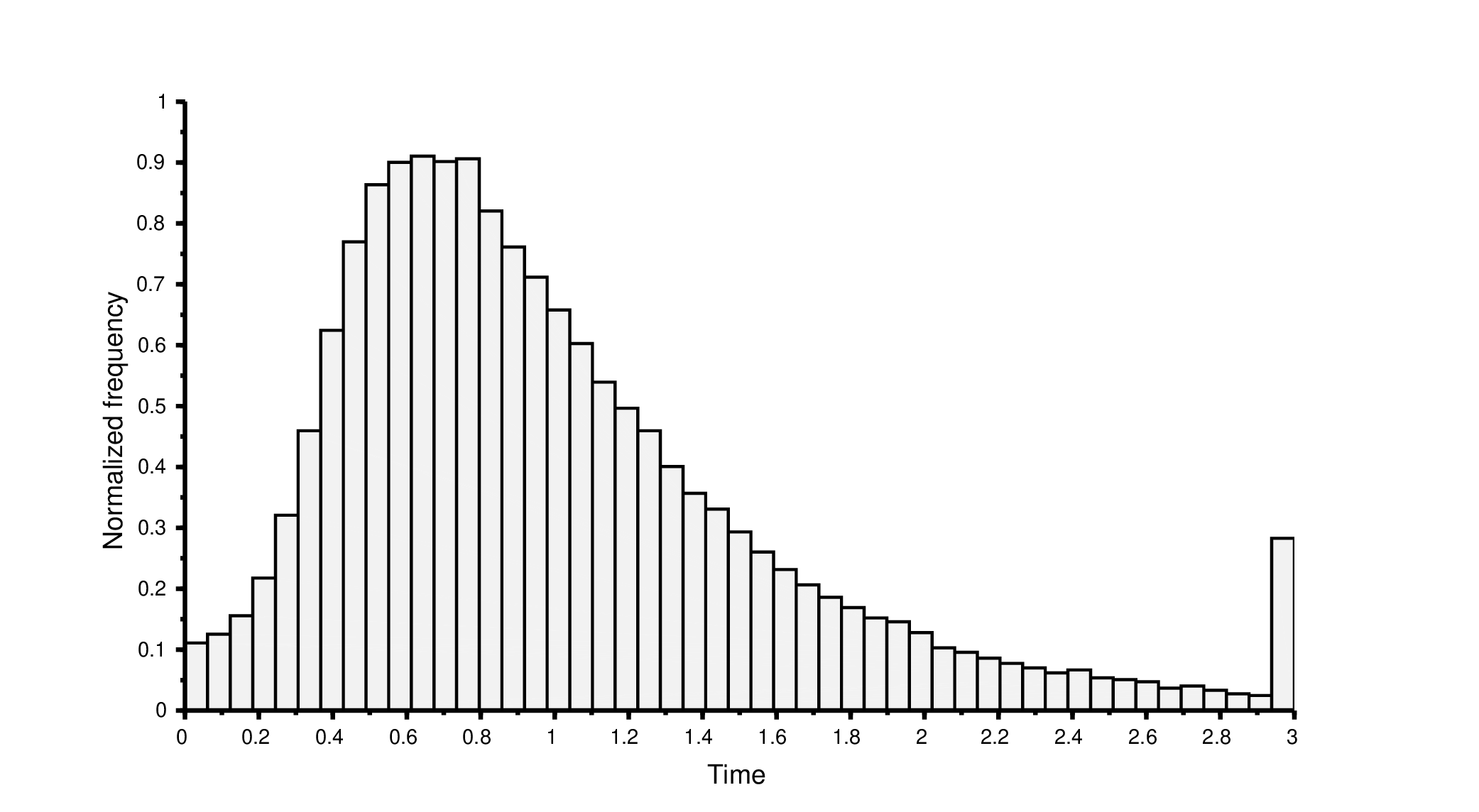}\hspace*{0.1cm}\includegraphics[width=9cm]{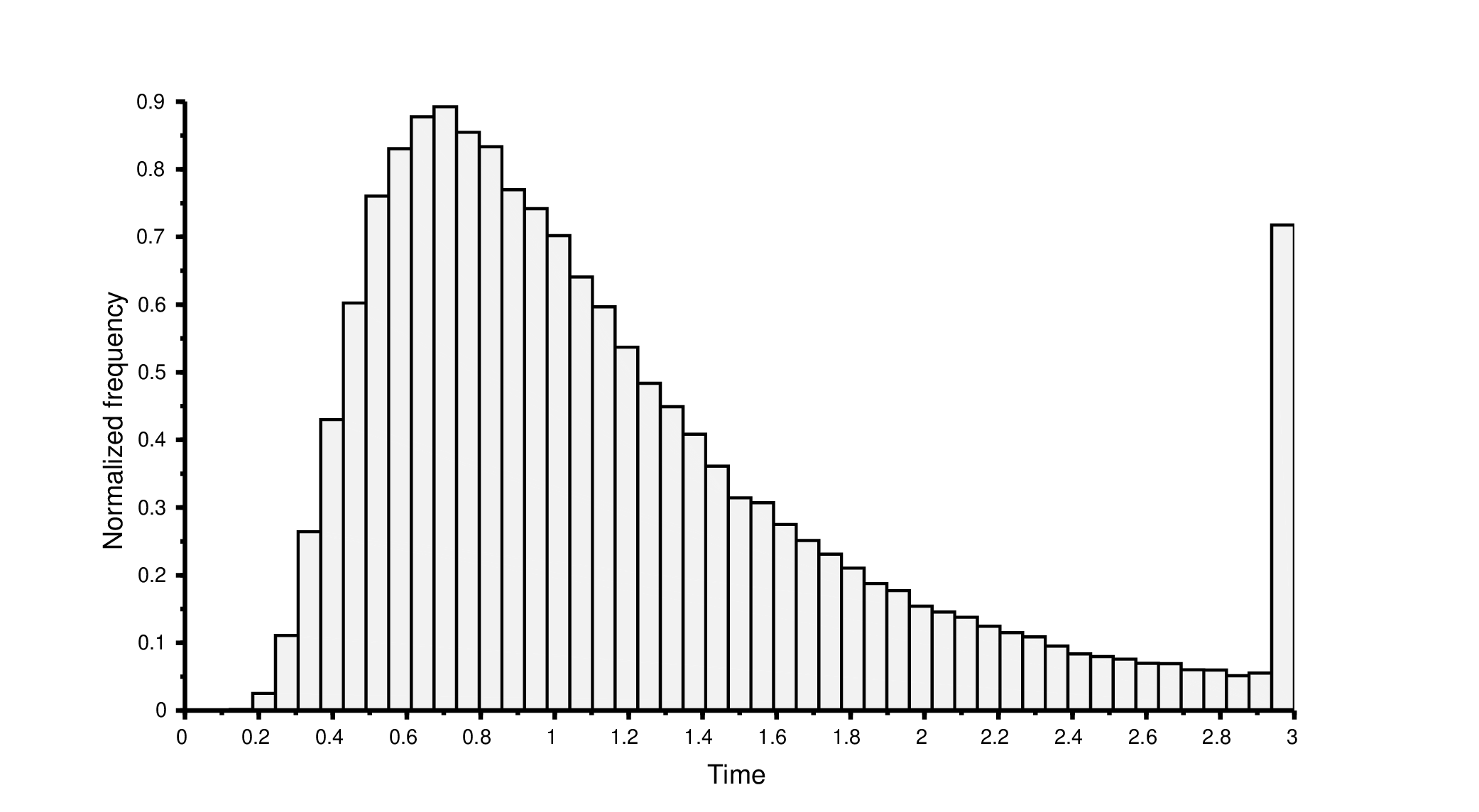}}
   \caption{\small Histograms of the stopping time $\tau_L\wedge \mathbb{T}$ for jump diffusion \eqref{exsin} with $j(t,y,z)=-z\sin(y)$. In this study, $y_0=-1$, $L=1$, $\mathbb{T}=3$, $\lambda=1$ and the size of the sample equals $100\,000$. The noise used for the jump generation corresponds to $\phi(t)=e^{-t}1_{\{t\ge 0\}}$ (left) or $\phi(t)=2\, 1_{[-1/4, 1/4]}(t)$ (right). We normalize the frequencies on the $y$-axis so that the area under the histogram is equal to one.}
   \label{fig:image1}
\end{figure}

The first assumption concerns the regularity of $\alpha(t,y)=\mu(t,y)$. The second and third assumptions  require the computation of bounds for both functions $\gamma(t,y)$ and $\beta(t,y)$ defined by \eqref{eq:def:beta-gamma}. In this study, we observe that:
\begin{equation*}
\beta(t,y) = \int^y_0 (2 + \sin(x))\, dx = 2y +1 - \cos(y) \leq 2L + 2 =: \beta_+,\quad \forall y\le L.
\end{equation*}
and:
\begin{equation*}
0\le \gamma(t,y) = \frac{(2+\sin(y))^2 + (2+\sin(y))'}{2} = \frac{(2+\sin(y))^2 + \cos(y)}{2} \leq 5 =: \kappa,\quad \forall y\le L.
\end{equation*}
The jump function can also be described as follows. We choose $j(t,y,z)=-z\sin(y)$. The histograms corresponding to the distribution of the first passage time $\tau_L\wedge \mathbb{T}$ are shown in Figure \ref{fig:image1}. 
%
As described in Section 2, the algorithm introduced in this study is based on rejection sampling that unfortunately  requires a long computation time. Several random variables related to Brownian motion are used as basic components: Brownian first passage times and values of the Brownian motion conditioned to stay under level $L$. These proposals are either rejected or accepted; then, the selected components are involved in the generation of the objective diffusion overrun. Thus, the ratio between accepted and rejected basic variates is a good efficiency indicator.      
The sample associated with the left figure requires approximately $377$ sec and its average acceptance rate is $1/580$; the right one requires approximately $150$ sec and the average acceptance rate is $1/190$.

As already explained in Remark \ref{rem:eff}, the efficiency of the algorithm depends on both parameters $\kappa$ and $\beta_+$: we must choose them to be as small as possible. 
These two examples show that the time consumption strongly depends on the jump measure.
A precise quantitative description of the efficiency would be cumbersome because the effects of all diffusion coefficients  $\mu$, $\sigma$, $\lambda$ and $\phi$ intermingle in a complex way.

Even if our numerical approach requires a long consumption time, it has the definitive advantage of being exact: no approximation error is introduced. Thus, the challenge for comparison of exact simulation and approximation methods is not feasible; users choose the best tool based on the objective to be achieved. We finally show the application of exact simulation to unbounded time domain as in the next example.

\subsubsection*{Example of jump diffusion that satisfies $\tau_L<\infty$.}
We now conclude the numerical illustrations with a jump diffusion starting in $y_0<L$ and  satisfying $\tau_L<\infty$. We introduce a stochastic process that satisfies Assumptions \ref{assu:Wee1} and \ref{assu:Wee2bis}, which ensure the finiteness of $\tau_L$ (see Theorem \ref{thm:Wee}), along with Assumptions \ref{assum20}, \ref{assum24} and \ref{assum25}. All these assumptions permit the use of Algorithm $({\rm JD})^{y_0,L}$ to determine the stopping time $\tau_L$.
We consider $(Y_t)_{t\ge 0}$ the solution of \eqref{exsin} between two consecutive jump times. The jumps are associated with a time-homogeneous function $j(y,v)=(L+1-y)v$ as follows:
\[
Y_{T_i}=Y_{T_i-}+j(Y_{T_i-},\xi_i),\quad i\in\mathbb{N}.
\]
This particular jump function drives the stochastic process towards the threshold $L$.
We recall that $T_i=\sum_{k=1}^i E_k$, where $E_k$ are exponentially distributed random variables with an average of  $1$. $(\xi_i)_{i\ge 1}$ is a sequence of independent uniformly distributed variates with density $\phi(v)=1_{[0,1]}(v)$. This model satisfies the announced assumptions quite easily. Let us just present the arguments 
that permit verification of Assumption \ref{assu:Wee2bis}. First, we observe that, for any $r>1$, $v\in]0,1[$ and $y\le -r$:
\begin{align}
 \ln(1-v)\le \ln\left(\frac{\vert y + j(y+L+r,v) \vert}{\vert y \vert}\right)= \ln\left(\frac{\vert y(1-v)+(1-r)v \vert}{\vert y \vert}\right)\le \ln(1-\frac{v}{r}).
 \label{ineq:ess}
\end{align}
This inequality leads to condition \eqref{eq:assu:Wee21} because: 
\begin{align*}
\int_\mathbb{R} \left( \ln\left(\frac{\vert y + j(y+L+r,v) \vert}{\vert y \vert}\right)\right)^2\phi(dv)\le \int_0^1 (\ln(1-v))^2\, dv< \infty.
\end{align*}
Finally, for condition \eqref{condition}, we note that $y\alpha(y+L+r)\le 0$ for any $y\le -r$ and introduce the constant:
\[
\eta:=-\int_0^1 \ln\Big(1-\frac{v}{r}\Big)\,dv. 
\]
Then, the definition of $\eta$ and \eqref{ineq:ess} imply:
\[
\int_\mathbb{R} y^2 \ln\left(\frac{\vert y + j(y+L+r,v) \vert}{\vert y \vert}\right)\phi(dv)= \int_\mathbb{R} y^2 \ln\left(\frac{\vert y(1-v)+(1-r)v  \vert}{\vert y \vert}\right)\phi(dv)\le-\eta y^2.
\]
All conditions presented in Assumption \ref{assu:Wee2bis} are therefore satisfied. Thus, the first passage time $\tau_L$ is almost certainly finite. Algorithm  $({\rm JD})^{y_0,L}$ generates a sample of this stopping time: the histogram in Figure \ref{fig:taufini} describes the probability distribution of the random variables. The generation of a sample of size 100 000 requires a processing time of approximately 90 
sec when $y_0=-1$ and 1 000 sec 
when $y_0=-3$. All coding was performed in the C++ programming language. This long processing time is strongly related to the nature of the algorithm: 
in most cases, algorithms based on rejection sampling require larger run times than approximation methods. 
Indeed, if the proposal distribution is far from the objective in some sense, the number of iterations must be large: typically, the proposal random object is rejected. 

As in the preceding example, we can focus on the number of basic components accepted or rejected by the algorithm. These basic components are Brownian first passage times and
values of the Brownian motion conditioned to stay under level L. Here, the average ratio between accepted and rejected basic variates equals approximately $1/180$  for the sample associated with Figure \ref{fig:taufini} (left) and $1/1610$ for Figure \ref{fig:taufini} (right).

The great advantage of the method is to avoid any approximation error: the outcome distribution coincides with the distribution of the first crossing time. A challenging topic is to suggest and study acceleration methods for such algorithms. For example, Herrmann and Zucca \cite{herrmann2020exact} proposed an acceleration procedure inspired by multiarmed bandits in the continuous diffusion framework.   
\begin{figure}[h]
   \centerline{\includegraphics[width=10cm]{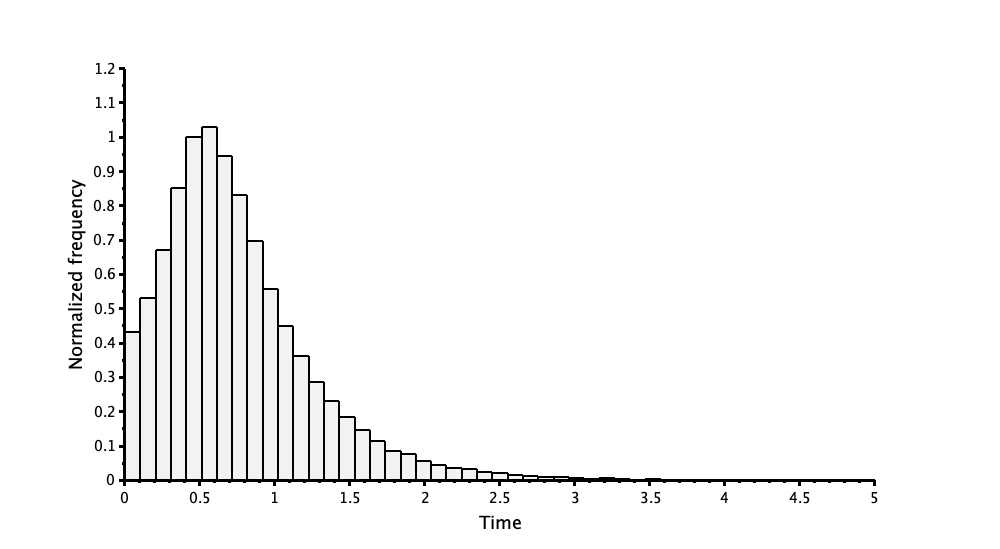}\hspace*{-1.5cm}
   \includegraphics[width=10cm]{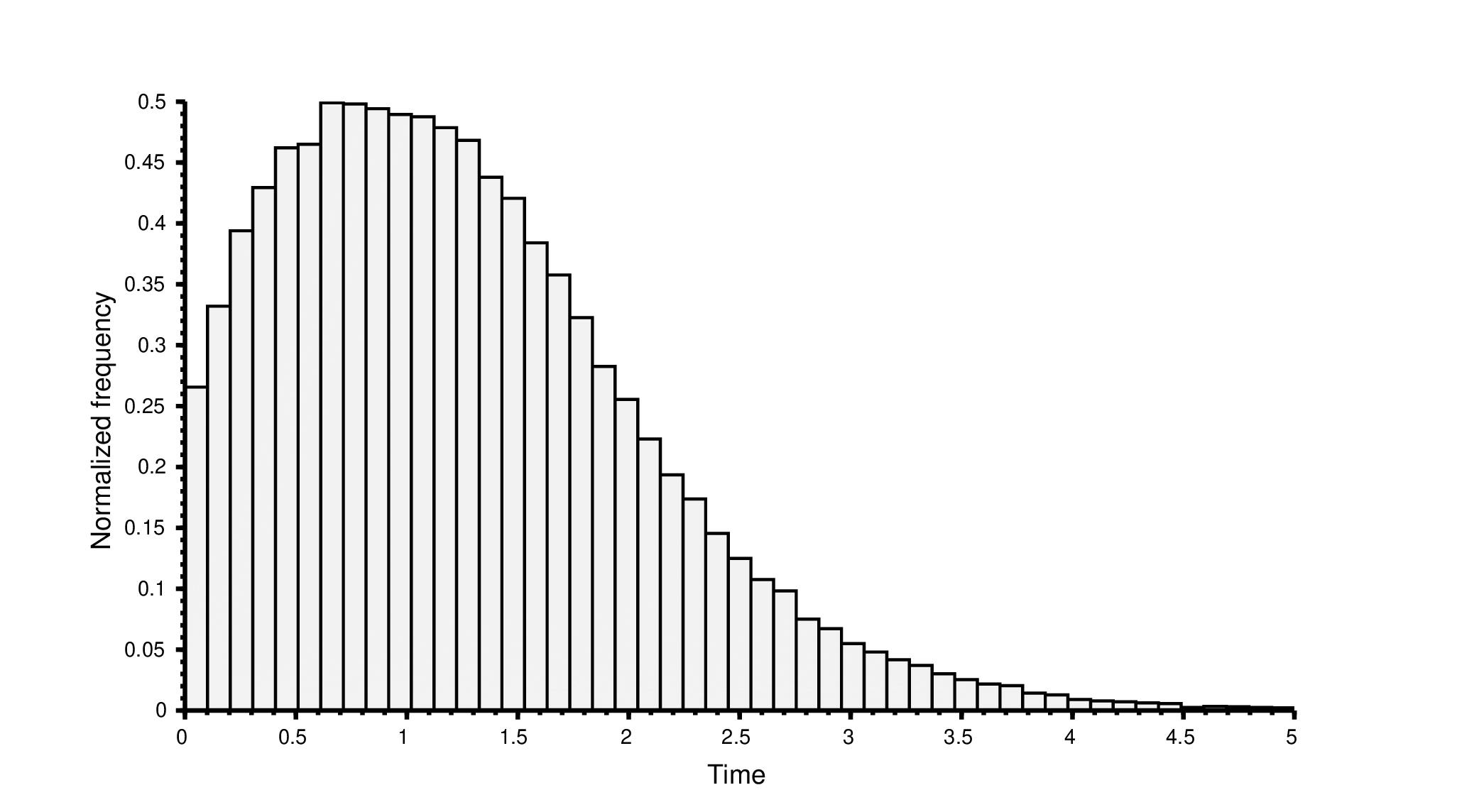}}
   \caption{\small Histograms of the stopping time $\tau_L$ for the jump diffusion \eqref{exsin} with $j(t,y,z)=(L+1-y)z$. In this study, $L=1$, the size of the sample equals $100\,000$ and $y_0=-1$ (left) or $y_0=-3$ (right). The noise used for the jump generation corresponds to $\phi(t)= 1_{[0,1]}(t)$. We normalize the frequencies on the $y$-axis so that the area under the histogram is equal to one.}
   \label{fig:taufini}
\end{figure}

\subsection*{Conclusion.}
We propose a new algorithm that permits us to generate the  time required by the jump diffusion process $(X_t)_{t\ge 0}$ to overrun a given threshold $L$. It is an exact generation without any approximation error. The approach is based both on the Markov property of the process and on rejection sampling. The probability that leads the decision of acceptance or rejection depends on the Girsanov transformation relating the Brownian paths to the paths of  diffusion. Such an approach can be adapted to time-dependent thresholds $(L_t)_{t\ge 0}$; thus we it introduce an auxiliary diffusion $Y_t:=X_t-(L_t-L_0)$. The stopping time $\tau_L$ then corresponds to the first time the jump diffusion $(Y_t)_{t\ge 0}$ overruns level $L_0$. The procedure may be adapted to the first exit time of an interval $I$. In \cite{herrmann2020exact1}, Herrmann and Zucca proposed an exact simulation of first exit times in the continuous  framework. Another method for stopped continuous diffusion was also proposed by Casella and Roberts \cite{casella-roberts}. Each of these two methods may represent the foundation stone of a generalised statement: we aim to write an efficient algorithm in the jump diffusion context.

\vspace{0.5cm}
\noindent {\bf Acknowledgement:} The authors would like to thank all referees for their interesting and constructive remarks, which improved this paper markedly.

\appendix
\section{Rejection sampling: fundamental results}
All algorithms belonging to the so-called \emph{exact simulation} framework essentially use rejection sampling. 
The primary properties of such a procedure used for the generation of random variables in many different situations can be described as follows. In most cases, the probability density of the 
variate is explicitly known and the algorithm is therefore succinctly written. In the context of diffusion processes, the algorithms used for the generation of path dependent variables are more sophisticated; they are nevertheless based on the same key idea. 

We let $(\Omega,\mathcal{F},\mathbb{P})$ be a probability space. We introduce a sequence of independent and identically distributed random experiments, and the $n$-th experiment permits the construction of an $\mathcal{F}$-measurable random couple $(Z_n,W_n)$ valued in $\mathbb{R}\times\{0,1\}$. The second marginal $W_n$ is a Bernoulli random variable that specifies if the experiment is successful. We then successively observe the experiments and focus the proposed attention on the variable $Z_n$ associated with the first success: 
\begin{equation}
\label{def:succes}
\mathcal{N}:=\inf\{n\ge 1:\ W_n=1\}\quad\mbox{and}\quad Z:=Z_\mathcal{N}.
\end{equation}
The rejection sampling corresponds exactly to a repetition of experiments until a successful result can be observed. The following statement (marginal modification of Theorem 3.2 in \cite{devroye}) can therefore be applied to the exact simulation.

\begin{proposition} \label{prop:rejet} We let $\psi$ any nonnegative Borel measurable function. The distribution of the random variable $Z$ defined by \eqref{def:succes} is characterised by:
\begin{equation}
\label{eq:chara}
\mathbb{E}[\psi(Z)]=\mathbb{E}[\psi(Z)|\mathcal{N}=1]=\frac{\mathbb{E}[\psi(Z)1_{\{\mathcal{N}=1\}}]}{\mathbb{P}(\mathcal{N}=1)}.
\end{equation}
\end{proposition}
\begin{proof} We consider all possible values of $\mathcal{N}$ which of course is geometrically distributed.
\begin{align*}
\mathbb{E}[\psi(Z)]&=\sum_{n\ge 1}\mathbb{E}\Big[\psi(Z)1_{\{\mathcal{N}=n\}}\Big]=\sum_{n\ge 1}\mathbb{E}\Big[\psi(Z_n)1_{\{W_1=0,\,W_2=0,\,\ldots,\, W_{n-1}=0,\,W_n=1\}}\Big].
\end{align*}
We denote the average of the Bernoulli random variable $W_1$ by $p$. Because the experiments are independent and identically distributed, we obtain:
\begin{align*}
\mathbb{E}[\psi(Z)]&=\mathbb{E}\Big[\psi(Z_1)1_{\{W_1=1\}}\Big]\sum_{n\ge 1}(1-p)^{n-1}=\frac{1}{p}\,\mathbb{E}\Big[\psi(Z)1_{\{\mathcal{N}=1\}}\Big].
\end{align*}
The identity \eqref{eq:chara} is a simple consequence of the property $p=\mathbb{P}(W_1=1)=\mathbb{P}(\mathcal{N}=1)$.
\end{proof}
Thus, the rejection sampling permits handling with conditional distributions. We also consider the cost of such an algorithm. As already mentioned, the number of experiments is geometrically distributed with an average of $1/\mathbb{P}(W_1=1)$. However, each experiment can also produce a particular cost: time-consumption of the experiment, number of computations, number of loops, number of variates. We now denote the (random) cost of the $n$-th experiment by $\mathcal{C}_n$, which depends on $W_n$, the success or failure of the corresponding experiment.  The total cost of the rejection sampling should therefore be related to both the number of Experiments $\mathcal{N}$ and the cost of each experiment $(\mathcal{C}_n)_{1\le n\le \mathcal{N}}$. Wald's identity leads to the following equation (see, for instance, Theorem 3.5 in \cite{devroye}):
\begin{equation}\label{cout}
\mathbb{E}\Big[\sum_{n=1}^\mathcal{N}\mathcal{C}_n\Big]=\mathbb{E}[\mathcal{C}_1]\mathbb{E}[\mathcal{N}],
\end{equation}
 and because $\mathcal{N}$ is geometrically distributed, we have $\mathbb{E}[\mathcal{N}]=1/\mathbb{P}(\mathcal{N}=1)$. The total cost is therefore precisely determined. 
 
\begin{comment}

\begin{proposition}\label{prop:wald} The average total cost of the rejection algorithm is given by
\begin{equation}\label{cout}
\mathbb{E}\Big[\sum_{n=1}^\mathcal{N}\mathcal{C}_n\Big]=\frac{\mathbb{E}[\mathcal{C}_1]}{\mathbb{P}(\mathcal{N}=1)},
\end{equation}
where $\mathcal{C}_1$ is the cost of the first experiment and $\mathcal{N}$ is the number of independent experiments observed until the algorithm stops.
\end{proposition}
\begin{proof} We recall that $W_n$ is a Bernoulli distributed random variable describing the success of the $n$-th experiment. Using Wald's identity, we obtain
\begin{align*}
\mathbb{E}\Big[\sum_{n=1}^\mathcal{N}\mathcal{C}_n\Big]&=\mathbb{E}\Big[\sum_{n=1}^{\mathcal{N}-1}\mathcal{C}_n\Big]+\mathbb{E}[\mathcal{C}_\mathcal{N}]=\mathbb{E}[\mathcal{N}-1]\,\mathbb{E}[\mathcal{C}_1|W_1=0]+\mathbb{E}[\mathcal{C}_1|W_1=1]\\
&=\Big(\frac{1}{\mathbb{P}(W_1=1)}-1\Big)\frac{\mathbb{E}[\mathcal{C}_1 1_{\{W_1=0\}}]}{\mathbb{P}(W_1=0)}+\frac{\mathbb{E}[\mathcal{C}_1 1_{\{W_1=1\}}]}{\mathbb{P}(W_1=1)}=\frac{\mathbb{E}[\mathcal{C}_1]}{\mathbb{P}(W_1=1)}.
\end{align*}
In order to conclude, we note that $\mathbb{P}(W_1=1)=\mathbb{P}(\mathcal{N}=1)$.
\end{proof}
\end{comment}

\bibliographystyle{plain}
\bibliography{biblithese}
\end{document}